\documentclass[10pt]{article}
\usepackage{amsfonts}
\usepackage{amsmath,amsthm,amscd,amssymb,mathrsfs,setspace,mathtools}
\usepackage{latexsym,epsf,epsfig}
\usepackage{color}
\usepackage[hmargin=1in,vmargin=1in]{geometry}
\usepackage{hyperref}
\usepackage{appendix}
\usepackage{cite}
\usepackage{caption}
\newcommand\bb[1]{\mathbf{#1}}
\newcommand{\ds}{\displaystyle}

\newcommand{\divg}{{\rm div}}

\newcommand{\bu}{\mathbf u}

\theoremstyle{plain}
\newtheorem{theorem}{Theorem}[section]
\newtheorem{lemma}[theorem]{Lemma}

\newtheorem{definition}[theorem]{Definition}
\newtheorem{assumption}{Assumption}
\theoremstyle{remark}
\newtheorem{remark}{Remark}[section]
\numberwithin{equation}{section}
\numberwithin{theorem}{section}
\numberwithin{remark}{section}
\numberwithin{assumption}{section}
\numberwithin{condition}{section}
\usepackage{float}
\usepackage[T1]{fontenc} 

\usepackage{tikz}
\usetikzlibrary{arrows.meta,decorations.pathmorphing,positioning}

\begin{document}
	
	\title{Steady weak solutions to an inflow/outflow driven \\ compressible fluid-structure interaction problem}
	
\date{}

\author{%
\begingroup
\footnotesize 
\setlength{\tabcolsep}{6pt}
\renewcommand{\arraystretch}{1.05}
\begin{tabular}{c}
	\begin{tabular}{c@{\hspace{1.6em}}c@{\hspace{1.6em}}c}
		\textbf{Boris Muha} & \textbf{\v{S}\' arka Ne\v{c}asov\'a} & \textbf{Milan Pokorn\'y} \\
		\it Faculty of Sciences, University of Zagreb & \it Czech Academy of Sciences & \it {Faculty of Mathematics and Physics,}  Charles University \\
		Zagreb, HR & Prague, CZ & Prague, CZ \\
		\texttt{borism@math.hr} & \texttt{matus@math.cas.cz} & \texttt{pokorny@karlin.mff.cuni.cz}
	\end{tabular}
	\\[1.0cm]
	\begin{tabular}{c@{\hspace{1.6em}}c}
		\textbf{Justin T.\ Webster} & \textbf{Sr\dj{}an Trifunovi\'{c}} \\
		\it UMBC & \it Faculty of Sciences, University of Novi Sad \\
		Baltimore, USA & Novi Sad, RS \\
		\texttt{websterj@umbc.edu} & \texttt{srdjan.trifunovic@dmi.uns.ac.rs}
	\end{tabular}
\end{tabular}
\endgroup
}
\maketitle

\begin{abstract}
\noindent
We study a stationary 3D/2D fluid-structure interaction problem between an elastic structure described by the linear plate equation and a fluid described by the compressible Navier-Stokes equations with hard-sphere pressure and inflow/outflow boundary data. This problem is motivated by wind-tunnel configuration and by the need for physically relevant steady states about which compressible flow-plate dynamics can be linearized.

The main difficulty in the analysis is the lack of uniform estimates, both for approximate and weak solutions. In particular, the fixed-point construction for approximate solution yields a density estimate
depending on approximate parameter, while the pressure estimate for the weak solution is only finite and non-quantifiable. As a result, large pressure loads can drive outward volume growth, while low pressure regions may lead to contact and therefore domain degeneration. This necessitates a novel approach based on a Lipschitz \emph{domain-correction} (barrier) mechanism that provides a framework in which solutions can be constructed without volume blow-up or degeneration of the domain. Constrained by the possibly very large fluid pressure load, our main result is the existence of a weak solution for a sufficiently large plate stiffness.

\vskip.15cm

\noindent {\em Keywords}: fluid-structure interaction, compressible Navier-Stokes, stationary weak solutions, hard-sphere pressure, inflow/outflow, linear plate, mathematical aeroelasticity
\vskip.15cm
\noindent {\em 2010 AMS}: 74F10, 76N10, 35Q35, 35D30, 76N15, 74K20  \end{abstract}

\section{The Model}\label{sec:model}

Here we study a steady interaction problem between a linear elastic plate and a compressible viscous fluid with both inflow and outflow (see \cite{PiaseckiPokorny2014}), suggesting a type of wind-tunnel configuration. The fluid reference domain is the unit cube
$$\mathscr{O}_0:= [0,1]^3.$$
The left and the right sides of the cube
$$\Sigma_{left}:= (0,1)\times\{0\}\times (0,1), \quad \Sigma_{right} = (0,1)\times\{1\}\times (0,1)$$
contain the inflow and outflow regions
$$\Sigma_{in} \subset\subset \Sigma_{left}, \quad \Sigma_{out} \subset\subset \Sigma_{right}. $$
The bottom side
$$\Sigma_{bot}:= (0,1)\times (0,1)\times\{0\}$$
houses a flat, deformable plate at equilibrium. More precisely 
$$\Gamma\subset\subset \Sigma_{bot}$$
is a Lipschitz domain and the vertical deformation of the plate defined on $\Gamma$ is described by a scalar displacement ~ $w:\Gamma \to \mathbb{R}$ which satisfies the linear plate equation
\begin{eqnarray}
\kappa\Delta^2 w=-S^w \mathbf{f}_{fl}\cdot \mathbf{e}_z, \quad \text{ on } \Gamma. \label{plate:eq}
\end{eqnarray}
Above, $\kappa>0$ is the elastic stiffness coefficient, $S^w=\sqrt{1+ |\nabla w|^2}$ is the Jacobian of the Eulerian-to-Lagrangian transformation of the plate, and $\mathbf{f}_{fl}$ is the force applied by the fluid onto the plate with $\mathbf{e}_z:=(0,0,1)$.  The plate is taken to be clamped on its smooth boundary, $\partial \Gamma$, i.e.
\begin{equation}\label{plateBC}
w =\partial_{\nu}w = 0,~~\text{ on }~~\partial \Gamma, 
\end{equation}
where $\nu$ is the outer normal vector on $\partial\Gamma$. 
The graph of $w$ is denoted as
\begin{eqnarray*}
\Gamma^w:=\{(x,y,w(x,y)): (x,y)\in \Gamma\}.
\end{eqnarray*}

The fluid fills the region determined by the plate and the reference domain (see Figure \ref{fig:dom} below)
\begin{eqnarray*}
\mathscr{O}(w):= \Big[\big((0,1)\times(0,1)\setminus \overline{\Gamma}\big)\times (0,1)\Big]\cup \Big\{(x,y,z): (x,y)\in\Gamma,~ w(x,y)<z<1\Big\}.
\end{eqnarray*}

\begin{figure}[h!]
\centering

\includegraphics[width=\textwidth]{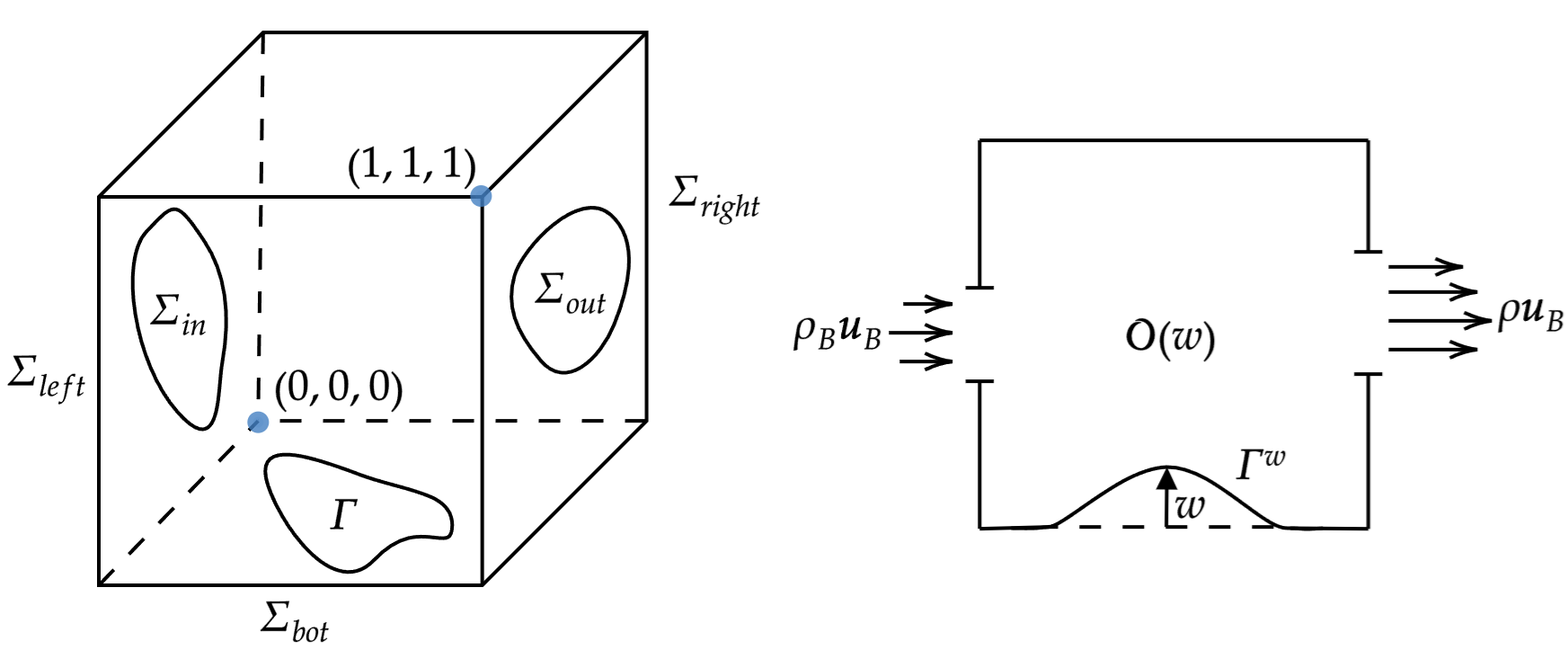}
\caption{Examples of a configuration domain $\mathscr{O}_0$ (left) and a section of a physical domain $\mathscr{O}(w)$ (right).}
\label{fig:dom}
\end{figure}

\noindent
The fluid is described by a density $\rho : \mathscr{O}(w)\to \mathbb{R}_{+}$ and a velocity $\bu : \mathscr{O}(w) \to \mathbb{R}^3$, with the pair obeying the steady compressible Navier-Stokes equations:
\begin{eqnarray}
\begin{cases}
	\divg (\rho \bu)=0  \\
	\divg  (\rho\bu \otimes \bu)+\nabla p(\rho)-\divg  \mathbb{S}(\nabla \bu)  = 0
\end{cases} \quad \text{ in } \mathscr{O}(w). \label{fluid:sys}
\end{eqnarray}
The pressure is {\em barotropic}, i.e., $p=p(\rho)$, and the viscous stress tensor is given by
\begin{eqnarray*}
\mathbb{S}(\nabla \bu):=\mu(\nabla\bu+(\nabla\bu)^T) + \lambda {\rm div}\bu \mathbb{I}, \quad \mu,\lambda>0. 
\end{eqnarray*}

\bigskip

We define, as data, the prescribed boundary velocity $\bb{u}_B: \Sigma_{in}\cup\Sigma_{out} \to \mathbb{R}^3$
and the inflow density $\rho_B:\Sigma_{in} \to \mathbb{R}_+$. Thus, the fluid boundary conditions are
\begin{equation}\label{FluidBC}
\begin{cases}
	\bu = \bu_B,&\quad \text{ on } \Sigma_{in}\cup\Sigma_{out},  \\
	\bu = 0, &\quad \text{ on } \partial \mathscr{O}(w)\setminus (\Sigma_{in}\cup \Sigma_{out}), \\
	\rho = \rho_{B}, & \quad \text{ on } \Sigma_{in}.
\end{cases} 
\end{equation}
We shall require that
\begin{eqnarray}
\begin{cases}\bu_B\cdot \bb{n} >0,& \text{ on } \Sigma_{out},\\
	\bu_B\cdot \bb{n} <0,& \text{ on } \Sigma_{in}, \end{cases} \label{FluidBC:followup}
	\end{eqnarray}
	where $\bb{n}$ is outward normal unit vector to the fluid domain. The dynamic coupling condition reads
	\begin{eqnarray}
\mathbf{f}_{fl}(x,y)&=&\big[(-p(\rho)\mathbb{I}+ \mathbb{S}(\nabla \bu))\mathbf{n}(w)\big](x,y,w(x,y)),\quad \text{ on } \Gamma, \label{dync}
\end{eqnarray}
where $\mathbf{n}(w):=\frac1{\sqrt{1+ |\nabla w|^2}}(-\nabla w, 1)$ is the outer normal unit vector to the displaced plate over the region $\Gamma$.

\bigskip

Finally, we consider the hard-sphere hypotheses, introduced in \cite{FN2018}. 
\begin{assumption}\label{ass1} The pressure $p$ satisfies the following:
\begin{itemize}
	\item There exists $\bar{\rho}>0$ such that ~$\ds \lim_{\rho\to\bar{\rho}}p(\rho)=+\infty$,
	\item $p\in C[0,\bar{\rho})\cap C^1(0,\bar{\rho})$, $p(0)=0$, $p'(\rho)\geq 0$, $\rho\in (0,\bar{\rho})$.
\end{itemize}
\end{assumption}

\section{Weak solution and main result}\label{sec:main}
We now introduce the concept of the weak solution for the configuration of interest. Our notions are motivated by \cite{FN2018}, but see also \cite{milan,book}. 
\begin{definition}\label{weak:sol}
For a given $\bb{u}_B\in C_0^2(\overline{\Sigma_{in}}\cup\overline{\Sigma_{out}})$ satisfying $\eqref{FluidBC:followup}$ and $\rho_B\in C(\Sigma_{in})$ with $0<\rho_B<\overline{\rho}$, we say that $(\rho,\bu,w)\in L^\infty( \mathscr{O}(w))\times H^1( \mathscr{O}(w))\times \big[H_0^2(\Gamma)\cap H^{\frac72}(\Gamma)\big]$ is a solution to \eqref{plate:eq}--\eqref{dync} if:
\begin{enumerate}
	\item $0\leq \rho<\overline{\rho}$, and $p(\rho) \in L^2(\mathscr{O}(w))$;
	\item $\bb{u} = \bb{u}_B$ on $\Sigma_{in}\cup\Sigma_{out}$ and $\bb{u}=0$ on $\partial\mathscr{O}(w)\setminus(\Sigma_{in}\cup\Sigma_{out})$ in the sense of traces;
	\item The continuity equation
	\begin{eqnarray*}
		\int_{ \mathscr{O}(w)}\rho\bb{u} \cdot \nabla \varphi = \int_{\Sigma_{in}} \rho_B\bu_B\cdot \bb{n} \varphi  \label{cont:weak}
	\end{eqnarray*}
	holds for all $\varphi \in C^\infty(\overline{ \mathscr{O}(w)})$ such that $\varphi_{|\Sigma_{out}} = 0$;
	\item The coupled momentum equation
	\begin{eqnarray*}
		&&\int_{ \mathscr{O}(w)} \rho \bu\otimes \bu: \nabla\boldsymbol{\varphi} + \int_{ \mathscr{O}(w)} p(\rho) \divg  \boldsymbol{\varphi}- \int_{ \mathscr{O}(w)}\mathbb{S}(\nabla \bu):\nabla\boldsymbol{\varphi} - \int_\Gamma \kappa\Delta w \Delta \psi = 0
	\end{eqnarray*}
	holds for every $\boldsymbol\varphi\in C_c^\infty(\mathscr{O}(w) \cup \Gamma^{w})$ and $\psi \in C_c^\infty(\Gamma)$ such that $\boldsymbol{\varphi}_{|\Gamma^{w}}  = \psi \bb{e}_z$.
\end{enumerate}
\end{definition}

\begin{remark}
Note that for every $\boldsymbol\varphi\in C_c^\infty(\mathscr{O}(w))$, one also has
\begin{eqnarray*}
	\int_{ \mathscr{O}(w)} p(\rho) \divg  \boldsymbol{\varphi}- \int_{ \mathscr{O}(w)}\mathbb{S}(\nabla \bu):\nabla\boldsymbol{\varphi}  = - \int_{ \mathscr{O}(w)} \rho \bu\otimes \bu: \nabla\boldsymbol{\varphi} = \int_{ \mathscr{O}(w)} (\rho \bu \cdot \nabla \bu)\cdot \boldsymbol{\varphi},
\end{eqnarray*}
by the continuity equation, so $\nabla p(\rho) -\divg \mathbb{S}(\nabla \bu)\in L^{\frac32}( \mathscr{O}(w))$. Combining this fact with $p(\rho)\mathbb{I} - \mathbb{S}(\nabla \bb{u}) \in L^{2}(\mathscr{O}(w))$ we obtain that $$\gamma_{|\Gamma^w}\left(\left[p(\rho)\mathbb{I} - \mathbb{S}(\nabla \bb{u})\right]\bb{n}(w) \right)\in W^{-\frac23,\frac32}(\Gamma),$$ and one can then decouple the plate equation from the coupled momentum equation above:
\begin{eqnarray*}
	\int_\Gamma \kappa\Delta w \Delta \psi = \langle S^w\gamma_{|\Gamma^w}\left[p(\rho)\mathbb{I} - \mathbb{S}(\nabla \bb{u})\right]\bb{n}(w) ,\psi \rangle_{W^{-\frac23,\frac32}(\Gamma)\times W^{\frac23,3}(\Gamma)} 
\end{eqnarray*}
for every $\psi \in H_0^2(\Gamma)$. Here $\gamma_{|\Gamma^w}$ denotes the \emph{Lagrangian} (pull-back) trace induced by the boundary
parametrization $$\Phi:\Gamma\to\Gamma^w,~~\Phi(X)=(X,w(X)).$$ Such Lagrangian/ALE trace operators are standard in the weak
formulation of moving-boundary FSI; see, e.g., \cite{Grandmont2008,ChambolleDesjardinsEstebanGrandmont2005,MuhaCanic2013}.
This regularity statement notably differs from the \emph{evolutionary} FSI case; in addition to the aforementioned references, see also \cite{Breit,berlin11} for representative formulations.
\end{remark}
\vskip.5cm

Our main result concerns the existence of a weak solution in the presence of inflow/outflow data for a sufficiently large stiffness $\kappa$. The stiffness threshold depends on the data, albeit in an indirect and non-quantifiable way, as we explain below. The main result is stated as follows:
\begin{theorem}[\textbf{Main result}]\label{main}
Let Assumption \ref{ass1} be in force. Let $\bb{u}_B\in C_0^2(\overline{\Sigma_{in}}\cup\overline{\Sigma_{out}})$ satisfying $\eqref{FluidBC:followup}$ and $$\int_{\Sigma_{in}\cup\Sigma_{out}} \bb{u}_B\cdot \bb{n}\geq 0,$$ and assume $\rho_B\in C(\Sigma_{in})$ with $0<\rho_B<\overline{\rho}$. Then, there is a lower threshold $\kappa_0>0$ such that for every $\kappa\geq \kappa_0$, there exists a weak solution $(\rho,\bb{u},w)$ in the sense of Definition \ref{weak:sol}.
\end{theorem}

Before providing commentary on our main theorem, we first emphasize the fundamental mechanisms that makes  stationary
weak theory delicate in this fluid-structure setting. Inflow/outflow conditions are physical for this class of problems, but the available pressure control is
geometry-driven and does not provide a bound on solutions controlled by the problem data. Consequently, {\bf there is nothing intrinsic to the coupled
model that precludes \emph{geometric degeneration} of the unknown domain $\mathscr O(w)$}. Excessively large pressure loads
may drive the plate outward and lead to {volumetric blow-up} of $\mathscr{O}(w)$;  low pressure (cavitation) effects, relative to the elastic response, may allow the plate to drift toward the rigid wall and produce near contact. In the latter case, relevant Korn/Poincar\'e/Bogovski\u{\i}
constants can deteriorate and destroy a priori estimates. The hard-sphere pressure law as in \cite{FN2018} enforces an upper density constraint, but, by itself, does not
rule out these  degeneracies. For this reason, we view some additional form of
\emph{constraint} or {\em intrinsic barrier} as essential for producing a physical stationary weak solution. 

In order to avoid fundamentally modifying the structure of the equations, we here
realize such control through the stiffness coefficient $\kappa$, together with a Lipschitz
\emph{domain-correction} device (a barrier method) that restores \emph{uniform} geometric control at the
approximate level. The combination thereof  yields uniform estimates needed for compactness and limit passage for sufficiently large stiffness values.

\medskip
Now we outline the central points concerning our main result.
\begin{enumerate}
\item \textbf{Pressure estimate and total fluid mass.}
In the framework of weak solutions for the problem at hand, intrinsically controlling the total mass seems to be out of reach. Mass is constrained only by
the (unknown) volume and the maximal density, as dictated by the hard-sphere pressure but not in any uniform way that
is usable in a fixed-point argument. The available pressure estimate is obtained  contradiction (as in \cite{FN2018})
and provides  a finite but non-quantifiable bound.  As a result,  large pressure loads may  drive excessive deformations, but assumptions on total mass or total volume of the domain indirectly constrain the maximal plate displacement, which is a priori unknown. This is a principal reason why the stiffness $\kappa$ must be
be sufficiently large to provide a structural counter-balance in the stationary weak setting.

\item \textbf{Construction of solutions and domain correction.}
The approximate problem is supplemented by artificial  dissipation for the density, the pressure is approximated,
and the fluid domain is regularized. Approximate solutions are obtained via the Schaefer fixed point theorem
and may, exhibit large plate displacements and  unfavorable geometries.
To prevent degeneration at the approximate level, and to 
constrain the deformable boundary, we introduce a Lipschitz \emph{domain correction}---a barrier method. This yields a family of approximate solutions that may initially involve a mismatch between the ``real'' and corrected domains. The final step
is to show that for sufficiently large stiffness, $\kappa$, the correction becomes inactive. This is
proved through a contradiction argument in the style of \cite[Section 7.1]{FN2018}. This yields the
existence of a threshold $\kappa_0$ such that for all $\kappa\ge \kappa_0$ the original coupled problem admits
a weak solution. The value $\kappa_0$ depends on the data, but not in an explicit manner,
reflecting again the non-quantitative nature of the pressure control in the hard sphere setting.

\item \textbf{Boundary data.}
The boundary velocity $\bb{u}_B$ may be large (while belonging to $C_0^2(\Sigma_{in}\cup\Sigma_{out})$).
The inflow density $\rho_B\in C(\Sigma_{in})$ is bounded by the maximal density, as dictated by the hard-sphere
law, but the boundary pressure $p(\rho_B)$ can be arbitrarily large. In this sense, the main theorem here is a \emph{large
	data}  result---see  \cite[Remark 2.2]{FN2018}. The high regularity imposed on the boundary data is used primarily to construct
extensions/lifts inside and outside of the variable domain in a compatible way  with the stationary theory.

\item \textbf{Hard-sphere pressure law.}
A singular hard-sphere pressure law is by now common in the compressible fluid literature. 
Analytically, it differs from the standard isentropic law, where $p(\rho)\sim \rho^\gamma$, but it is currently the
state-of-the-art framework for which stationary inflow/outflow results on a fixed domain are available
\cite{FN2018}. The induced upper density bound is crucial for the renormalization of the continuity equation,
which in turn is essential to obtain compactness of the density in the final limiting procedure (to remove artificial dissipation).\footnote{ The existence of non-steady weak solutions of the hard pressure case in a bounded domain with no-slip boundary conditions were studied  by Feireisl and Zhang \cite[Section 3]{MR2646821}. The case of  general inflow/outflow was investigated  by Choe, Novotn\' y and Yang \cite{MR3912678}. }

\item \textbf{Applications.} We stress the broader motivation here: much of the compressible flow-plate literature on aeroelastic flutter and
stability of flow-plate systems is based on linearization about a steady FSI configuration. Theorem~\ref{main} provides a
mathematically justified and generally nontrivial steady state in an inflow/outflow driven compressible setting. The existence of a weak solution here
supplies a concrete base state for subsequent linearizations and related spectral/semigroup analyses as in the work of Avalos and Geredeli et al., see for instance \cite{AvalosGeredeli2016-MMAS,AvalosGeredeli2020-JEE,AGW2018-DCDSB,Geredeli2021-MMAS,KLT2018}.
\end{enumerate}

\begin{remark}[Generalizations]\label{rem:generalizations}
We briefly describe several directions in which the present existence result can be extended.
We do not pursue these items here, since the main ideas (uniform geometric control, coupled compactness,
and the removal of regularizations) are made transparent in the model/configuration as described above.

\begin{itemize}
	\item {\bf Geometry.}
	The choice of the unit cube $\mathscr{O}_0=[0,1]^3$ is made for concreteness and  its natural
	interpretation representing a wind-tunnel. The analysis extends to general bounded reference domains
	in $\mathbb{R}^3$, provided the inflow/outflow portions $\Sigma_{in}$ and $\Sigma_{out}$ are compactly contained
	in open subsets of sufficiently smooth boundary components. The latter ensures that the prescribed boundary
	data admit ``lifts'', while the remaining portion of the boundary supports the no-slip condition, cf.\ \cite{Choe2018,FN2018,Choe2019}. A salient  requirement for the coupled
	problem is that the deformable  boundary (the plate region) be separated from the inflow/outflow
	regions, so that the variable interface does not interfere with  prescribed fluxes and so  associated
	extension/Korn/Bogovski\u{\i} constants are controlled uniformly.
	
	A further geometric variant is to replace the flat plate by a reference surface with curvature (e.g.\ a linear
	shell). At the level of stationary solutions there is no obstruction to such a nontrivial equilibrium geometry, and the
	arguments here can be adapted, provided the elastic operator remains elliptic with comparable
	coercivity and trace control. Related FSI settings with curved elastic boundaries (in evolutionary regimes) can be
	found, for instance, in \cite{Breit,cr-full-karman,berlin11}.
	
	\item {\bf Plate nonlinearity.}
	For clarity we take the structure to satisfy a \emph{linear} clamped plate equation, and we use the stiffness
	parameter $\kappa$ as the principal mechanism balancing  large pressure loads. In this stationary weak
	framework we do not expect that incorporating a ``large deflection'' superlinearity can, by itself,
	rule out infinite-volume limits. This is again due to the fact that pressure
	control in the hard-sphere inflow/outflow setting is not quantitative in terms of the data---see
	\cite{FN2018,Lions}. On the other hand, our approach is compatible with many standard \emph{semilinear} plate
	nonlinearities, such as those of Kirchhoff, Berger, or von K\'arm\'an type; we require that a plate nonlinearity is locally Lipschitz on the natural plate
	energy space and admits a ``good" potential energy. Such hypotheses are classical
	in mathematical aeroelasticity and plate theory; see, e.g., \cite{supersonic,karmanplates,CLDW2016-MESA}.
	
	\item {\bf Other barrier mechanisms.}
	The domain-correction device invoked is a barrier method  enforcing an a priori geometric constraint
	at the approximate level. While we implement this via a stiffness threshold
	$\kappa\ge \kappa_0$ (and an associated Lipschitz control of the deformable boundary), other physically relevant
	constraining mechanisms could resolve the interplay between geometry-dependent constants, very weak pressure bounds, and strong plate
	loading. A natural candidate is a geometric smallness parameter tied to the \emph{relative size} of the deformable
	boundary portion, as $\Gamma$ compared to the overall fluid chamber. At present, however, such dependence is not easily quantifiable, as the
	relevant Korn/Poincar\'e/Bogovski\u{\i} constants and critical embeddings are not tracked sharply; thus, we opt to work
	directly with stiffness as a clean and transparent mathematical barrier mechanism in the present paper.
	
	\item {\bf Source force.}
	Additional source forces acting on the fluid or  plate can also be directly included; however, as the inflow/outflow
	mechanism is of primary interest here, we omit such forcing for clarity. This is in line with \cite{FN2018}---see Remark 2.4 for more details about the accommodation of internal sources. 
\end{itemize}
\end{remark}

\section{Literature review}

Fluid-structure interaction (FSI) models  are  ubiquitous in engineering applications, especially those in which an elastic component is exposed to an adjacent or surrounding flow of fluid. Canonical examples include aeroelastic structures (aircraft skins, panels, and control surfaces), propulsion systems, bridge and building components in high winds, and a variety of ducted-flow configurations---see \cite{dowell1,CLDW2016-MESA,KLT2018} and references therein for examples. In  such settings, the interaction may not be adequately described by incompressible or inviscid flow regimes. Indeed, one may  need to account for both \emph{compressible} effects (e.g., due to non-negligible Mach number) as well as  \emph{viscosity}. The latter is especially true when boundary layers, shear, or nontrivial inflow/outflow profiles become prominent. See the references \cite{BA62,bolotin,dowell1,chorin-marsden} for the relevant general background and motivation on such flows and related FSIs. From the mathematical point of view, it is precisely the presence of additional nonlinear fluid equations, coupling, and compressible flow constraints that make the analysis of compressible FSI  delicate in several ways not seen for the incompressible counterpart.

In the last two decades, the mathematical literature on fluid/flow-plate interactions has grown significantly. Depending on the application and motivation for the study, the literature can be divided into two major directions: (i) geometrically nonlinear and (ii) geometrically linear interaction problems.  

The nonlinear interaction class of problems does not neglect changes of the fluid domain and, as such, are highly nonlinear and technical in the treatment of time-varying, unknown domains. In particular, the possibility of domain degeneration may dictate the lifespan of the solution, while the regularity and many available tools in the analysis are highly affected by the unknown shape of the fluid domain. Moreover, even in configurations where one aims for weak solutions, a moving-domain setting forces one to confront the stability of geometric constants (e.g., Poincar\'e/Korn/Bogovski\u{\i} constants) under boundary deformations, the latter of which is a persistent obstacle in compactness arguments \cite{ChambolleDesjardinsEstebanGrandmont2005,Boulakia2005CR,MuhaCanic2013,MuhaCanic2019}.

Linear interaction problems \cite{dcds,CLDW2016-MESA} neglect the global domain change and, as such, are more suitable for stability analysis and long-time behavior. Nonlinearity may be included in the elastic structure or flow equations, {\em but the interface is taken to be fixed}. Such analyses have developed along several complementary lines: semigroup well-posedness and stability of linearized models \cite{supersonic,lasweb,AGW2019-JMAA,Chu2013-comp}, frequency-domain and spectral techniques for decay rates \cite{AvalosGeredeli2016-MMAS,AvalosGeredeli2020-JEE, Geredeli2021-MMAS}, and nonlinear/dynamical-systems analyses leading to attractors and qualitative stabilization phenomena \cite{ChuRyz2011,karmanplates}. In the \emph{compressible} regime, there are many works that focus on \emph{evolution} models, often in situations where simplifying assumptions on the flow permit a reduction to wave-type dynamics (potential flow) or to linearized compressible PDE models. We refer to the surveys \cite{berlin11,dcds,CLDW2016-MESA} and the references therein for a broader perspective. Within the compressible viscous setting, a central motivating contribution is \cite{Chu2013-comp}, which first studies the time-dependent interaction of a plate with a linearized compressible viscous fluid in a cavity configuration.

From \cite{Chu2013-comp}, the most closely aligned work with that at hand  are recent investigations by
Avalos et al. in which \emph{linearized} compressible
Navier-Stokes equations (NSE) are coupled to a nonlinear plate and analyzed via semigroup methods. In particular, in \cite{AGW2018-DCDSB,AGW2019-JMAA}, the compressible NSE are linearized about a prescribed steady state and the resulting coupled dynamics are treated on the natural finite-energy phase space. These works address, among other issues, the nontrivial operator structure induced by compressibility and non-dissipative (and potentially de-stabilizing) coupling at the interface, and they provide a framework in which questions of stability can be formulated and analyzed \cite{AvalosGeredeli2020-JEE,AvalosGeredeli2016-MMAS}. 

However, the aforementioned program requires a physically meaningful stationary configuration about which the coupled dynamics can be linearized; yet the existence of such steady states, driven by inflow/outflow data in a compressible setting, is far from obvious. One must ask: 
\vskip.15cm
\noindent\textit{Do physically relevant steady states for compressible flow-plate interactions exist, and under what conditions?}
\vskip.15cm
\noindent In the incompressible regime, stationary FSI problems have been 
considered previously: in \cite{Grandmont2002}, existence of regular 
solutions is established for a three-dimensional steady 
Navier--Stokes--elasticity system under a small data assumption, 
while \cite{Calisti2023} addresses shape sensitivity for a 
two-dimensional stationary Stokes--elasticity configuration, also 
in the small-load regime. In the compressible viscous setting, the existence theory for the stationary fluid equations with inflow/outflow data is both recent and technically demanding, and only available in the hard-sphere framework 
\cite{Lions,FN2018,Choe2019}. The compressible flow-plate model studied in our paper thus represents a natural continuation of the current body of work as it lies 
at the intersection of two demanding analytical programs: stationary compressible NSE with general boundary conditions, and nonlinear boundary-coupled FSI with an a priori unknown interface geometry. For this problem, we are able to show existence of a weak solution under the condition of sufficiently large plate stiffness, giving a direct answer to the above question.

\section{Construction of weak solutions on a corrected domain}
In this section, the goal is to construct a weak solution on a domain which is a priori restricted. There are multiple reasons for this choice, as mentioned in the introduction. Namely, the estimates which are available for the fluid pressure seem to be non-quantifiable, giving us a finite pressure but bounded by a constant which does not explicitly depend on data. Moreover, the fixed-point construction of the approximate solutions will not be able to ensure smallness of the plate in any norm. Both of the above issues may yield a degeneration, as the plate can reach the rigid boundary, or blow-out (deform downwards unboundedly) resulting in infinite domain volume. Also, the Lipschitz norm of the plate may be large,  increasing the pressure and, consequently, other estimates. 

Thus, we introduce a correction function
\begin{eqnarray*}
f_{cor}(x):= \begin{cases}
	1,& \text{ for } 0\leq x \leq \frac14, \\
	4x,& \text{ for } x > \frac14,
\end{cases}
\end{eqnarray*}
and the corrected domain displacement is defined as
\begin{eqnarray*}
[w]_{cor}:= \frac{w}{f_{cor}(\|w\|_{C^{0,1}(\Gamma)})}.
\end{eqnarray*}
with the corrected domain is then naturally defined as $\mathscr{O}([w]_{cor})$. Note that $\mathscr{O}([w]_{cor})$ cannot degenerate no matter how large $w$ in $C^{0,1}(\Gamma)$ norm is, since $[w]_{cor}\leq \frac14$. We now introduce the concept of a weak solution on a corrected domain:

\begin{definition}\label{cor:dom:def}
For a given $\bb{u}_B\in C_0^2(\overline{\Sigma_{in}}\cup\overline{\Sigma_{out}})$ satisfying $\eqref{FluidBC:followup}$ and $\rho_B\in C(\Sigma_{in})$ with $0<\rho_B<\overline{\rho}$, we say that $(\rho,\bu,w)\in L^\infty( \mathscr{O}([w]_{cor}))\times H^1( \mathscr{O}([w]_{cor}))\times \big[H_0^2(\Gamma)\cap H^{\frac72}(\Gamma)\big]$ is a weak solution on a corrected domain if:
\begin{enumerate}
	\item $0\leq \rho<\overline{\rho}$,~ $p(\rho) \in L^2(\mathscr{O}([w]_{cor}))$;
	\item $\bb{u} = \bb{u}_B$ on $\Sigma_{in}\cup\Sigma_{out}$ and $\bb{u}=0$ on $\partial\mathscr{O}([w]_{cor})\setminus(\Sigma_{in}\cup\Sigma_{out})$ in the sense of traces;
	\item The continuity equation
	\begin{eqnarray*}
		\int_{ \mathscr{O}([w]_{cor})}\rho\bb{u} \cdot \nabla \varphi = \int_{\Sigma_{in}} \rho_B\bu_B\cdot \bb{n} \varphi  \label{cont:weak*}
	\end{eqnarray*}
	holds for all $\varphi \in C^\infty(\overline{ \mathscr{O}([w]_{cor})})$ such that $\varphi_{|\Sigma_{out}} = 0$;
	\item The coupled momentum equation
	\begin{eqnarray*}
		&&\int_{ \mathscr{O}([w]_{cor})} \rho \bu\otimes \bu: \nabla\boldsymbol{\varphi} + \int_{ \mathscr{O}([w]_{cor})} p(\rho) \divg  \boldsymbol{\varphi}- \int_{ \mathscr{O}([w]_{cor})}\mathbb{S}(\nabla \bu):\nabla\boldsymbol{\varphi} - \int_\Gamma \kappa\Delta w \Delta \psi = 0
	\end{eqnarray*}
	holds for every $\boldsymbol\varphi\in C_c^\infty(\mathscr{O}([w]_{cor}) \cup \Gamma^{[w]_{cor}})$ and $\psi \in C_c^\infty(\Gamma)$ such that $\boldsymbol{\varphi}_{|\Gamma^{[w]_{cor}}}  = \psi \bb{e}_z$;
\end{enumerate}
\end{definition}
The goal of this section is the following foundational result:

\begin{theorem}\label{not:main}
Let $\bb{u}_B\in C_0^2(\Sigma_{in}\cup\Sigma_{out})$ satisfy $\eqref{FluidBC:followup}$ and $\int_{\Sigma_{in}\cup\Sigma_{out}} \bb{u}_B\cdot \bb{n}\geq 0$, let $\rho_B\in C(\Sigma_{in})$ with $0< \rho_B<\overline{\rho}$. Then, there exists a weak solution $(\rho,\bb{u},w)$ on a corrected domain in the sense of Definition \ref{cor:dom:def} such that
$$\| \bu \|_{H^1(\mathscr{O}([w]_{cor})))}  \leq C$$
where constant $C$ only depends on the geometry, fluid viscosity constants $\mu,\lambda$ and the boundary velocity $\bb{u}_B$.
\end{theorem}

\subsection{Preliminaries: minimal and maximal domain, fluid velocity extension}\label{prel:sec}
First, let $\Gamma'$ be a smooth domain such that $\Gamma\Subset\Gamma'\Subset \Sigma_{bot}$. Now, since $|[w]_{cor}|\leq \frac14$, we fix a function $w_{max}\in C_c^\infty(\Gamma')$ such that 
\begin{eqnarray*}
0\leq w_{max}\leq \frac12& \quad  &\text{ on } \Gamma', \\
w_{max}=\frac12& \quad &\text{ on } \Gamma, 
\end{eqnarray*}
which defines the minimal and maximal domain
\begin{eqnarray*}
\mathscr{O}_{min}:= \mathscr{O}(-w_{max}), \quad \mathscr{O}_{max}:= \mathscr{O}(w_{max}),
\end{eqnarray*}
where the definition of $\mathscr{O}(w)$ was extended to functions defined on $\Gamma'$ instead of $\Gamma$. Finally, noticing that that since 
$\text{dist}(\Sigma_{in},\partial\Sigma_{left})>0$ and $\text{dist}(\Sigma_{out},\partial\Sigma_{right})>0$, it is immediate that there is a smooth connected domain $\tilde{\mathscr{O}}\subset \mathscr{O}_{min}$ such that:
$$\Sigma_{in}, \Sigma_{out} \subset \partial \tilde{\mathscr{O}}.$$

Now we are ready to introduce the velocity extension:
\begin{lemma}\label{ext:op}
Let $w\in C_0(\Gamma)$ such that $\|w\|_{C(\Gamma)}\leq \frac14$ and $\bu_B\in W^{2-\frac1q,q}(\Sigma_{in}\cup\Sigma_{out})$, $q>3$, satisfying $\eqref{FluidBC:followup}$. Then, for any given $\theta>0$, there exists a constant $C=C(\theta,\mathscr{O}_0)$ that blows up as $\theta\to 0$ and an extension of $\bu_B$ in $W^{1,\infty}(\mathscr{O}_{max})$ also denoted $\bb{u}_B$ such that
\begin{eqnarray*}
	&&\|\bu_B\|_{W^{1,\infty}(\mathscr{O}_{max})}\leq C\| \bu_B\|_{W^{2-\frac1q,q}(\Sigma_{in}\cup \Sigma_{out})},\\[.2cm]
	&&\bb{u}=0 \text{ on } \partial\mathscr{O}_{max}\setminus(\Sigma_{in}\cup\Sigma_{out}) \text{ in the sense of traces}, \\[.2cm]
	&&\divg ~\bu_B\geq 0 \text{ a.e. on } \mathscr{O}_{max},\quad \bu_B = 0 \text{ on }  \mathscr{O}_{max}\setminus \mathscr{O}_{min},  \\[.2cm]
	&&\int_{\mathscr{O}_{max}} |\mathbf{v} \cdot \nabla \mathbf{v} \cdot \bu_B| \leq \theta \| \nabla \mathbf{v}\|_{L^2(\mathscr{O}_{max})}^2 \quad \text{ for any } \quad \mathbf{v} \in H_0^1(\mathscr{O}_{max}).
\end{eqnarray*}
\end{lemma}
\begin{proof}
Since $\tilde{\mathscr{O}}$ is of $C^3$ regularity, by \cite[Proposition 3.4]{FN2018}, there exists the extension to $\tilde{\mathscr{O}}$ such that
\begin{eqnarray*}
	&& \|\bu_B\|_{W^{1,\infty}(\tilde{\mathscr{O}})} \leq C\|\bu_B\|_{W^{2,q}(\tilde{\mathscr{O}})}\leq C(\theta,\tilde{\mathscr{O}})\| \bu_B\|_{W^{2-\frac1q,q}(\partial \tilde{\mathscr{O}})},\\[.2cm]
	&&\divg ~\bu_B\geq 0 \text{ a.e. on } \tilde{\mathscr{O}}, \qquad \bu_B = 0 \text{ on }  \partial\tilde{\mathscr{O}}\setminus(\Sigma_{in}\cup\Sigma_{out}) \text{ in the sense of traces},\\[.2cm]
	&& \int_{\tilde{\mathscr{O}}} |\mathbf{v} \cdot \nabla \mathbf{v} \cdot \bu_B| \leq \theta \| \nabla \mathbf{v}\|^2_{L^2(\tilde{\mathscr{O}})} \quad \text{ for any } \quad \mathbf{v} \in H_0^1(\tilde{\mathscr{O}}).
\end{eqnarray*}
Since this extension vanishes on $\partial\tilde{\mathscr{O}}\setminus(\Sigma_{in}\cup\Sigma_{out})$, it can be extended by zero to $\mathscr{O}_{max}\setminus\tilde{\mathscr{O}}$, which clearly satisfies the same conditions; as such, the proof is finished.
\end{proof}

\begin{lemma}\label{extension}
Let $w\in C^{0,1}(\Gamma)\cap C_0(\Gamma)$ such that $\|w\|_{ C(\Gamma)}\leq \frac14$. Then, there exists an extension operator $A:W_0^{1-\frac1p,p}(\Gamma) \to W^{1,p}(\mathscr{O}(w))$ for $p>2$ or $A:W^{1-\frac1p,p}(\Gamma) \to W^{1,p}(\mathscr{O}(w))$ for $1<p\leq 2$ such that
\begin{eqnarray*}
	\gamma_{|\Gamma^w}A[f] &=& f, \\
	\|A[f]\|_{W^{1,p}(\mathscr{O}(w))} &\leq& C \|f\|_{W^{1-\frac1p,p}(\Gamma)},
\end{eqnarray*}
where $C$ only depends on $\|w\|_{ C^{0,1}(\Gamma)}$, and $A[f]=0$ on $\partial\mathscr{O}(w)\setminus \Gamma^w$ in the sense of traces.
\end{lemma}
\begin{proof}
Operator $A$ can be defined as $A[f]=r \bb{e}_z$ where $r$ is the solution to
\begin{eqnarray*}
	\begin{cases}
		\Delta r=0, & \text{ in } \left\{ (x,y,z): (x,y)\in \Gamma, w(x,y)<z<\frac12 \right\},\\
		r = f, & \text{ on } \Gamma^w,\\
		r = 0,  & \text{ on } \left(0,\frac12\right)\times\partial\Gamma\cup \left\{\frac12\right\}\times \Gamma,
	\end{cases}
\end{eqnarray*}
and then extended by zero to $\mathscr{O}(w)$. 
\end{proof}

\begin{lemma}\label{conv:lemma}
Let $w_n \to w$ in $C_0^{0,1}(\overline{\Gamma})$ such that $\|w_n\|_{C(\Gamma)}\leq \frac14$ and let $\|f_n\|_{W^{1,p}(\mathscr{O}(w_n))}\leq C$ for $p>\frac65$. Then, there exists $f\in L^2(\mathscr{O}(w))$ such that $f_n \chi_{|\mathscr{O}(w_n)}\to f\chi_{|\mathscr{O}(w)}$ in $L^2(\mathbb{R}^3)$, at least for a subsequence.
\end{lemma}
\begin{proof}
First, let $f_n,f$ be extended by zero to $\mathbb{R}^3$ which by imbedding satisfy $\|f_n\|_{L^q(\mathbb{R}^3)},\|f\|_{L^q(\mathbb{R}^3)} \leq C$, for some $\frac{3p}{3-p}=q>2$, and let $r=\frac2{1-2/q}$. Now, fix $K\Subset \mathscr{O}(w)$ such that $|\mathscr{O}(w) \setminus K|<(\varepsilon/2)^{\frac{r}2}$ and let $n_K$ be such that for all $n\geq n_K$ one has $K\Subset \mathscr{O}(w_n)$ and $|\mathscr{O}(w_n) \setminus K|<\varepsilon^{\frac{r}2}$. Note that since $\|f_n\|_{W^{1,p}(K)}\leq C$, we can fix a subsequence (not relabeled) such that $\int_K |f_n - f|^2 \to 0$ as $n\to \infty$. Therefore, for any $n\geq n_K$ one has
\begin{eqnarray*}
	&&\int_{\mathbb{R}^3} |f_n \chi_{|\mathscr{O}(w_n)}- f\chi_{\mathscr{O}(w)}|^2 \\
	&&\leq 2 \int_{\mathbb{R}^3} \chi_{|K}|f_n - f|^2+ 2 \int_{\mathbb{R}^3} |f_n|^2 | \chi_{|\mathscr{O}(w_n)}-\chi_{|K}|^2 + 2 \int_{\mathbb{R}^3} |f|^2| \chi_{|K}- \chi_{|\mathscr{O}(w)}|^2 \\
	&&\leq 2  \int_{K} |f_n - f|^2 +2\|f_n\|_{L^q(\mathbb{R}^3)}^2 \left(\int_{\mathbb{R}^3} |\chi_{|K}- \chi_{\mathscr{O}(w_n)}|^{r} \right)^{\frac2r}+2\|f\|_{L^q(\mathbb{R}^3)}^2 \left(\int_{\mathbb{R}^3} |\chi_{|K}- \chi_{\mathscr{O}(w)}|^{r} \right)^{\frac2r}\\
	&&\leq 2  \int_{K} |f_n - f|^2 + 3\varepsilon C 
\end{eqnarray*}
so
\begin{eqnarray*}
	\lim_{n\to \infty}\int_{\mathbb{R}^3} |f_n \chi_{|\mathscr{O}(w_n)}- f\chi_{\mathscr{O}(w)}|^2 \leq 3\varepsilon C.
\end{eqnarray*}
Letting $|\mathscr{O}(w)\setminus K| \to 0$ gives us the desired conclusion.    
\end{proof}

\subsection{The approximate problem}\label{sec:reg}
The approximate problem is defined as follows. Let $\varepsilon,\delta \in (0,1)$. There are two levels of approximation, as in \cite{FN2018}:
\begin{enumerate}
\item $\varepsilon$-level: approximate pressure and approximate domain;
\item $\delta$-level: artificial density dissipation.
\end{enumerate}

First, we regularize $\rho_B$ by $\rho_B^\delta\in C^1(\Sigma_{in})$ such that $\rho_B^\delta \to \rho_B$ in $C(\Sigma_{in})$ as $\delta\to 0$ and $\rho_B^\delta<\overline{\rho}$ for all $\delta$. Next, let us introduce the cut-off
\begin{eqnarray*}
T(\rho):= \begin{cases}
	0, & \text{ for } \rho \leq 0,\\
	\rho,& \text{ for } 0\leq \rho \leq \overline{\rho},\\
	\overline{\rho},& \text{ for } \rho \geq \overline{\rho},\\
\end{cases}
\end{eqnarray*}
and then the approximate pressure
\begin{eqnarray*}
p_{\varepsilon,\delta}(\rho):=p_\varepsilon(\rho)+\sqrt{\delta}\rho,
\end{eqnarray*}
where
\begin{eqnarray*}
p_\varepsilon(\rho):=\begin{cases}
	p(\rho) ,& \text{ if } 0\leq \rho \leq \overline{\rho}-\varepsilon,\\
	p(\overline{\rho}-\varepsilon)+ p'(\overline{\rho}-\varepsilon)(\rho - \overline{\rho}+\varepsilon ),& \text{ if }\rho > \overline{\rho}-\varepsilon.
\end{cases}
\end{eqnarray*}
Finally, since the proof will rely on a sufficient regularity of the fluid domain, we introduce a family of smooth connected convex approximate domains $\mathscr{O}_\varepsilon\subset \mathscr{O}_0$ which are uniformly Lipschitz such that:
\begin{eqnarray*}
\Sigma_{in}, \Sigma_{out}, \Gamma' \subset \partial \mathscr{O}_\varepsilon,\qquad \Gamma'\times(0,1)\subset \mathscr{O}_\varepsilon, \qquad |\mathscr{O}_0\setminus\mathscr{O}_\varepsilon|\to 0 \text{ as } \varepsilon\to 0, \qquad    \tilde{\mathscr{O}} \subset \mathscr{O}_\varepsilon \text{ for all } \varepsilon \in (0,1),
\end{eqnarray*}
where $\tilde{\mathscr{O}}$ was introduced at the beginning of Section \ref{prel:sec}, and the approximate domain 
$$\mathscr{O}_\varepsilon(w) := \mathscr{O}_\varepsilon \cap \mathscr{O}(w).$$
Next, from here onward, we fix the extension $\bu_B\in W^{1,\infty}(\mathscr{O}_{max})$ constructed in Lemma $\ref{ext:op}$ corresponding to a fixed small $\theta>0$ for which the estimate $\eqref{u:est}$ holds, and the space:
\begin{eqnarray*}
H_{B}^1( \mathscr{O}_\varepsilon([w]_{cor})):= \{\bu\in H^1( \mathscr{O}_\varepsilon([w]_{cor})): \bu - \bu_B \in H_0^1( \mathscr{O}_\varepsilon([w]_{cor}))\}.
\end{eqnarray*}
Finally, since we will work with elliptic regularity, our space for $w$ needs to ensure that the fluid domain is of $C^{1,1}$ regularity. Therefore, since $H^{s}(\Gamma)$ is imbedded into $C^{1,1}(\overline{\Gamma})$ for any $s>3$, we fix $a\in (0,1/2)$ and the space target space as $H_0^2(\Gamma)\cap H^{3+a}(\Gamma)$. We are ready to define:

\begin{definition}\label{app:sol}
We say that $(\rho,\bu,w)\in H^1( \mathscr{O}_\varepsilon([w]_{cor}))\times H_B^1( \mathscr{O}_\varepsilon([w]_{cor}))\times \big[H_0^2(\Gamma)\cap H^{3+a}(\Gamma)\big]$ is an approximate solution if:
\begin{enumerate}
	\item Damped continuity equation
	\begin{eqnarray*}
		&&\delta\int_{ \mathscr{O}_\varepsilon([w]_{cor})} \nabla\rho \cdot \nabla \varphi +  \delta\int_{ \mathscr{O}_\varepsilon([w]_{cor})} \rho \varphi  - \int_{ \mathscr{O}_\varepsilon([w]_{cor})} T(\rho)\bb{u} \cdot \nabla \varphi \\
		&&= -\int_{\Sigma_{in}} \rho_B^\delta\bu_B\cdot \bb{n} \varphi - \int_{\Sigma_{out}} T(\rho)\bu_B\cdot \bb{n} \varphi 
	\end{eqnarray*}
	holds for all $\varphi \in C^\infty(\overline{ \mathscr{O}_\varepsilon([w]_{cor})})$;
	\item The coupled momentum equation
	\begin{eqnarray*}
		&&\int_{ \mathscr{O}_\varepsilon([w]_{cor})} (T(\rho) \bu\otimes \bu): \nabla\boldsymbol{\varphi} + \int_{ \mathscr{O}_\varepsilon([w]_{cor})} p_{\varepsilon,\delta}(\rho) \divg  \boldsymbol{\varphi}- \int_{ \mathscr{O}_\varepsilon([w]_{cor})}\mathbb{S}(\nabla \bu):\nabla\boldsymbol{\varphi} \\
		&&- \delta  \int_{ \mathscr{O}_\varepsilon([w]_{cor})} \nabla(\rho \bu)\cdot \nabla\boldsymbol{\varphi} - \delta  \int_{ \mathscr{O}_\varepsilon([w]_{cor})} \rho \bu\cdot\boldsymbol{\varphi}- \int_\Gamma \kappa\Delta w \Delta \psi = 0
	\end{eqnarray*}
	holds for every $\boldsymbol\varphi\in C_c^\infty(\mathscr{O}_\varepsilon([w]_{cor}) \cup \Gamma^{[w]_{cor}})$ and $\psi \in C_c^\infty(\Gamma)$ such that $\boldsymbol{\varphi}_{|\Gamma^{[w]_{cor}}}  = \psi \bb{e}_z$.
\end{enumerate}
\end{definition}

\subsection{Solving the approximate problem}\label{sec:fixed}
\textit{Convention on notation}. From here onward, for any constant appearing in an estimate, the dependence on $\Gamma,\overline{\rho}, \mathscr{O}_{min},\mathscr{O}_{max}$ will not be mentioned since they are all fixed. On the other hand, the dependence on $\|\bb{u}_B\|_{W^{2-\frac2q,q}(\Sigma_{in}\cup\Sigma_{out})}$ or any lower order norm of $\bb{u}_B$ will be emphasized if the estimate is uniform, otherwise it will simply be denoted as $C=C(\bb{u}_B)$ (for example if it depends on $\delta^{-1}$). \\

In order to solve the problem by a fixed point, we need to fix a space. Since the fluid velocity space depends on domain displacement, let us define the desired space for which the fluid domain is fixed
\begin{eqnarray*}
V:= \big\{(\bb{u}, w)\in H_B^1(\mathscr{O}_{max})\times \big[H_0^2(\Gamma)\cap H^{3+a}(\Gamma)\big]: \bb{u} = 0 \text{ a.e. on } \mathscr{O}_{max}\setminus  \mathscr{O}_\varepsilon([w]_{cor})  \big\},
\end{eqnarray*}
and the operator
\begin{eqnarray*}
\mathcal{T}:&&  H_B^1(\mathscr{O}_{max})\times \big[H_0^2(\Gamma)\cap H^{3+a}(\Gamma)\big] \to  H_B^1(\mathscr{O}_{max})\times \big[H_0^2(\Gamma)\cap H^{3+a}(\Gamma)\big]\\
&& (\tilde\bu, \tilde{w}) \mapsto (\bu, w)
\end{eqnarray*}
where $w = w[\rho, \bb{u},\tilde{\bu},\tilde{w}]\in H_0^2(\Gamma)\cap H^{3+a}(\Gamma)$ is the solution to 
\begin{eqnarray}
&& \int_\Gamma \kappa\Delta w \Delta \psi\nonumber\\
&&=\int_{\mathscr{O}_\varepsilon([\tilde{w}]_{cor})} (T(\rho) \tilde{\bu}\otimes \tilde{\bu}): \nabla\boldsymbol{\varphi} + \int_{\mathscr{O}_\varepsilon([\tilde{w}]_{cor})} p_{\varepsilon,\delta}(\rho) \divg  \boldsymbol{\varphi}-\delta\int_{\mathscr{O}_\varepsilon([\tilde{w}]_{cor})}  \tilde{\bu}\otimes\nabla\rho: \nabla\boldsymbol{\varphi} \nonumber \\
&& -\int_{\mathscr{O}_\varepsilon([\tilde{w}]_{cor})}\mathbb{S}(\nabla \bu):\nabla\boldsymbol{\varphi} - \delta  \int_{\mathscr{O}_\varepsilon([\tilde{w}]_{cor})} \rho \nabla\bu:\nabla\boldsymbol{\varphi} - \delta  \int_{\mathscr{O}_\varepsilon([\tilde{w}]_{cor})} \rho \bu\cdot\boldsymbol{\varphi} \label{app:plate}
\end{eqnarray}
that holds for $\boldsymbol\varphi\in C_c^\infty(\mathscr{O}_\varepsilon([\tilde{w}]_{cor}) \cup \Gamma^{[\tilde{w}]_{cor}})$ and $\psi \in C_c^\infty(\Gamma)$ such that $\boldsymbol{\varphi}_{|\Gamma^{[\tilde{w}]_{cor}}}  = \psi \bb{e}_z$, then $\bb{u} = \bb{u}[\rho,\tilde{\bb{u}},\tilde{w}]\in H_B^1(\mathscr{O}_{\varepsilon}[\tilde w]_{cor})$ extended by zero to $\mathscr{O}_{max}\setminus  \mathscr{O}_\varepsilon([\tilde{w}]_{cor})$ is the solution to
\begin{eqnarray}
&&\int_{\mathscr{O}_\varepsilon([\tilde{w}]_{cor})}\mathbb{S}(\nabla \bu):\nabla\boldsymbol{\varphi}+\delta  \int_{\mathscr{O}_\varepsilon([\tilde{w}]_{cor})} \rho\nabla \bu\cdot \nabla\boldsymbol{\varphi} + \delta  \int_{\mathscr{O}_\varepsilon([\tilde{w}]_{cor})} \rho \bu\cdot\boldsymbol{\varphi} \nonumber \\
&&=-\delta\int_{\mathscr{O}_\varepsilon([\tilde{w}]_{cor})}  \tilde{\bu}\otimes\nabla\rho: \nabla\boldsymbol{\varphi} +\int_{\mathscr{O}_\varepsilon([\tilde{w}]_{cor})} (T(\rho) \tilde{\bu}\otimes \tilde{\bu}): \nabla\boldsymbol{\varphi} + \int_{\mathscr{O}_\varepsilon([\tilde{w}]_{cor})} p_{\varepsilon,\delta}(\rho) \divg  \boldsymbol{\varphi}  \label{app:fixed:mom}
\end{eqnarray}
that holds for every $\boldsymbol\varphi\in C_c^\infty(\mathscr{O}_\varepsilon([\tilde{w}]_{cor}))$, and finally $\rho = \rho[\tilde{\bb{u}},\tilde{w}] \in H^2(\mathscr{O}_\varepsilon([\tilde{w}]_{cor})$ is the solution to 
\begin{eqnarray}
&&\delta\int_{\mathscr{O}_\varepsilon([\tilde{w}]_{cor})} \nabla\rho \cdot \nabla \varphi +  \delta\int_{\mathscr{O}_\varepsilon([\tilde{w}]_{cor})} \rho \varphi  - \int_{\mathscr{O}_\varepsilon([\tilde{w}]_{cor})} T(\rho)\tilde{\bb{u}} \cdot \nabla \varphi \nonumber \\
&&= -\int_{\Sigma_{in}} \rho_B^\delta\bu_B\cdot \bb{n} \varphi - \int_{\Sigma_{out}} T(\rho)\bu_B\cdot \bb{n} \varphi \label{app:fixed:cont}
\end{eqnarray}
holds for all $\varphi \in C^\infty(\overline{\mathscr{O}_\varepsilon([\tilde{w}]_{cor})})$.\\

\noindent
\textbf{Well-definedness of operator}: The existence of unique solution $\rho\geq 0$ to the problem $\eqref{app:fixed:cont}$ follows by method of monotone operators and its uniqueness by \cite[Lemma 4.1]{FN2018}. Next, in order to solve $\eqref{app:fixed:mom}$, one can rewrite it as
\begin{eqnarray*}
&&\int_{\mathscr{O}_\varepsilon([\tilde{w}]_{cor})}\mathbb{S}(\nabla \bu - \nabla \bu_B):\nabla\boldsymbol{\varphi}+\delta  \int_{\mathscr{O}_\varepsilon([\tilde{w}]_{cor})} \rho(\nabla \bu-\nabla \bu_B)\cdot \nabla\boldsymbol{\varphi} + \delta  \int_{\mathscr{O}_\varepsilon([\tilde{w}]_{cor})} \rho (\bu-\bu_B)\cdot\boldsymbol{\varphi} \nonumber \\
&&=\int_{\mathscr{O}_\varepsilon([\tilde{w}]_{cor})} (T(\rho) \tilde{\bu}\otimes \tilde{\bu}): \nabla\boldsymbol{\varphi} + \int_{\mathscr{O}_\varepsilon([\tilde{w}]_{cor})} p_{\varepsilon,\delta}(\rho) \divg  \boldsymbol{\varphi} -\delta\int_{\mathscr{O}_\varepsilon([\tilde{w}]_{cor})}  \tilde{\bu}\otimes\nabla\rho: \nabla\boldsymbol{\varphi} \\
&& \quad - \int_{\mathscr{O}_\varepsilon([\tilde{w}]_{cor})}\mathbb{S}( \nabla \bu_B):\nabla\boldsymbol{\varphi}-\delta  \int_{\mathscr{O}_\varepsilon([\tilde{w}]_{cor})} \rho\nabla \bu_B: \nabla\boldsymbol{\varphi} - \delta  \int_{\mathscr{O}_\varepsilon([\tilde{w}]_{cor})} \rho\bu_B\cdot\boldsymbol{\varphi}
\end{eqnarray*}
for all $\boldsymbol\varphi\in C_c^\infty(\mathscr{O}_\varepsilon([\tilde{w}]_{cor}))$. Now, the unique solution $\bu - \bu_B$ follows by Lax-Milgram lemma, as the left-hand side constitutes a linear coercive operator in $H_0^1(\mathscr{O}_\varepsilon([\tilde{w}]_{cor}))$. Finally, the unique solution to $\eqref{app:plate}$ follows once again by Lax-Milgram lemma.
\bigskip

\noindent
\textbf{Compactness}: Let
\begin{eqnarray*}
(\tilde\bu, \tilde{w}) \in S_K:=\{( \bu, w) \in V: 
\|\bb{u}\|_{H^1(\mathscr{O}_\varepsilon([w]_{cor}))} + \|w\|_{H^{3+a}(\Gamma)}\leq K \}.
\end{eqnarray*}
Now, let us first point out that
\begin{eqnarray}
\|\tilde{w}\|_{C^{1,1}(\Gamma)}\leq C\|\tilde{w}\|_{H^{3+a}(\Gamma)} \leq CK, \label{C11}
\end{eqnarray}
and since $\tilde{w} = \nabla \tilde{w} = 0$ on $\partial\Gamma$, we have that the domain $\mathscr{O}_\varepsilon([\tilde{w}]_{cor})$ is of $C^{1,1}$ regularity. Next, testing $\eqref{app:fixed:cont}$ with $\rho$ gives us
\begin{eqnarray}
&&\delta \| \rho \|_{H^1(\mathscr{O}_\varepsilon([\tilde{w}]_{cor}))}^2+\int_{\Sigma_{out}} \underbrace{\rho T(\rho)\bu_B\cdot \bb{n}}_{\geq 0} \nonumber\\
&&= \int_{\mathscr{O}_\varepsilon([\tilde{w}]_{cor})} T(\rho) \tilde{\bb{u}}\cdot\nabla \rho -\int_{\Sigma_{in}} \rho \rho_B^\delta\bu_B\cdot \bb{n}   \nonumber\\ 
&&\leq C(\delta^{-1}) \| \tilde{\bb{u}}\|_{L^2(\mathscr{O}_\varepsilon([\tilde{w}]_{cor}))}^2 + C(\delta^{-1},\bb{u}_B) + \frac\delta2\| \rho \|_{H^1(\mathscr{O}_\varepsilon([\tilde{w}]_{cor}))}^2 \label{H1:rho}
\end{eqnarray}
where we estimated by the trace inequality
\begin{eqnarray*}
-\int_{\Sigma_{in}} \rho \rho_B^\delta\bu_B\cdot \bb{n} \leq \overline{\rho} \|\rho\|_{L^2(\Sigma_{in})}\|\bu_B\|_{L^2(\Sigma_{in})} \leq C(\delta^{-1},\bb{u}_B)+\frac\delta4 \|\rho\|_{H^1(\mathscr{O}_\varepsilon([\tilde{w}]_{cor}))}^2 .
\end{eqnarray*}
Now, since $-\rho_B^\delta\bb{u}_B+\rho \bb{u}_B \in W^{\frac13, \frac32}(\Sigma_{in}\cup\Sigma_{out})$ and $ \nabla \cdot (\rho \bb{u})\in L^{\frac32}( \mathscr{O}_\varepsilon([\tilde{w}]_{cor}))$, the unique solution $\tilde{\rho} \in W^{2,\frac32}( \mathscr{O}_\varepsilon([\tilde{w}]_{cor}))$ to the problem
\begin{eqnarray*}
\begin{cases}-\delta \Delta \tilde{\rho} + \delta \tilde{\rho} = \nabla \cdot (\rho \bb{u}),& \quad \text{ in } \mathscr{O}_\varepsilon([\tilde{w}]_{cor})\\
	\delta \partial_n \tilde{\rho} = 0,& \quad \text{ on } \partial \mathscr{O}_\varepsilon([\tilde{w}]_{cor})\setminus (\Sigma_{in}\cup \Sigma_{out}), \\
	\delta \partial_n \tilde{\rho} = -\rho_B^\delta\bb{u}_B \cdot \bb{n}+T(\rho)\bb{u}_B\cdot\bb{n},& \quad \text{ on } \Sigma_{in}, 
\end{cases}
\end{eqnarray*}
exists due to $C^{1,1}$ regularity of the domain \eqref{C11} and \cite[Theorem 2.4.2.7]{grisvard}, which has to coincide with $\rho$. Thus,
\begin{eqnarray*}
\|\nabla \rho\|_{L^3( \mathscr{O}_\varepsilon([\tilde{w}]_{cor}))} \leq C\| \rho\|_{W^{2,\frac32}( \mathscr{O}_\varepsilon([\tilde{w}]_{cor}))} \leq C(K,\delta^{-1},\bb{u}_B)
\end{eqnarray*}
so we can now improve the estimate due to $-\rho_B^\delta\bb{u}_B+\rho \bb{u}_B \in H^{\frac12}(\Sigma_{in}\cup\Sigma_{out})$
\begin{eqnarray}
&&\delta\|\rho\|_{H^{2}( \mathscr{O}_\varepsilon([\tilde{w}]_{cor}))} \nonumber\\
&&\leq C \|\nabla \rho\|_{L^3( \mathscr{O}_\varepsilon([\tilde{w}]_{cor}))}\|\bb{u}\|_{L^6( \mathscr{O}_\varepsilon([\tilde{w}]_{cor}))} + C \|-\rho \bb{u}_B+\rho_B^\delta \bb{u}_B\|_{H^{\frac12}( \mathscr{O}_\varepsilon([\tilde{w}]_{cor}))} \nonumber\\
&&\leq C(K,\delta^{-1},\bb{u}_B). \label{H2:rho}
\end{eqnarray}
Next, by choosing $\boldsymbol{\varphi}=\bb{u}-\bb{u}_B$ in $\eqref{app:fixed:mom}$ leads to
\begin{eqnarray*}
&&c\|\bb{u}-\bb{u}_B\|_{H^1(\mathscr{O}_\varepsilon([\tilde{w}]_{cor}))}^2 \\
&&\leq \int_{\mathscr{O}_\varepsilon([\tilde{w}]_{cor})}\mathbb{S}(\nabla \bu-\nabla \bb{u}_B):(\nabla\bb{u}-\nabla\bb{u}_B) + \delta \int_{\mathscr{O}_\varepsilon([\tilde{w}]_{cor})} \rho (|\nabla\bb{u}-\nabla \bu_B|^2+|\bb{u}-\bb{u}_B|^2)  \\
&& =\int_{\mathscr{O}_\varepsilon([\tilde{w}]_{cor})} (T(\rho) \tilde{\bu}\otimes \tilde{\bu}): (\nabla \bu-\nabla \bb{u}_B) + \int_{\mathscr{O}_\varepsilon([\tilde{w}]_{cor})} p_{\varepsilon,\delta}(\rho) \divg (\bb{u}-\bb{u}_B) -\delta\int_{\mathscr{O}_\varepsilon([\tilde{w}]_{cor})}  \tilde{\bu}\otimes\nabla\rho: (\nabla \bu-\nabla \bb{u}_B) \\
&& \quad - \int_{\mathscr{O}_\varepsilon([\tilde{w}]_{cor})}\mathbb{S}( \nabla \bu_B):(\nabla \bu-\nabla \bb{u}_B)-\delta  \int_{\mathscr{O}_\varepsilon([\tilde{w}]_{cor})} \rho\nabla \bu_B: (\nabla \bu-\nabla \bb{u}_B)- \delta  \int_{\mathscr{O}_\varepsilon([\tilde{w}]_{cor})} \rho\bu_B\cdot( \bu- \bb{u}_B)
\end{eqnarray*}
since $\bb{u}-\bb{u}_B \in H_0^1(\mathscr{O}_\varepsilon([\tilde{w}]_{cor}))$. Now, by using $\eqref{H2:rho}$ and the imbedding $\|\tilde{\bb{u}}\|_{L^6(\mathscr{O}_\varepsilon([\tilde{w}]_{cor}))} \leq C\|\tilde{\bb{u}}\|_{H^1(\mathscr{O}_\varepsilon([\tilde{w}]_{cor}))}$, one can easily control and absorb the terms on RHS to obtain
\begin{eqnarray}
\|\bb{u}\|_{H^1(\mathscr{O}_\varepsilon([\tilde{w}]_{cor}))} \leq C( K,\delta^{-1},\bb{u}_B). \label{uH1}
\end{eqnarray}
Noticing that $\bb{f}\mapsto -\nabla \cdot \mathbb{S}(\nabla \bb{f}) - \delta \rho \Delta \bb{f} + \delta \rho \bb{f}$ is uniformly elliptic in $\bb{f}$, while 
\begin{eqnarray*}
\divg (\nabla\rho \otimes \tilde{\bb{u}}), ~\nabla \rho \cdot \nabla \bb{u}  \in L^{\frac32}(\mathscr{O}_\varepsilon([\tilde{w}]_{cor})),
\end{eqnarray*}
one obtains by \cite[Theorem 2.4.2.5]{grisvard} and $C^{1,1}$ regularity of $\mathscr{O}_\varepsilon([\tilde{w}]_{cor})$ that $\bb{u}$ is in fact a strong solution satisfying 
\begin{eqnarray}
\|\bb{u}\|_{W^{2,\frac32}(\mathscr{O}_\varepsilon([\tilde{w}]_{cor}))} \leq C( K,\delta^{-1}, \bb{u}_B). \label{u:w232}
\end{eqnarray}
Finally, we can test $\eqref{app:plate}$ with $\boldsymbol{\varphi}= A[w]$ and $\psi = w$ where $A[w]$ is the extension defined in Lemma $\ref{extension}$ to obtain
\begin{eqnarray}
&&\int_\Gamma \kappa|\Delta w|^2  \nonumber\\
&&\leq\int_{\mathscr{O}_\varepsilon([\tilde{w}]_{cor})} (T(\rho) \tilde{\bu}\otimes \tilde{\bu}): \nabla A[w] +\int_{\mathscr{O}_\varepsilon([\tilde{w}]_{cor})} p_{\varepsilon,\delta}(\rho) \divg  A[w] \nonumber \\
&&\quad -\int_{\mathscr{O}_\varepsilon([\tilde{w}]_{cor})}\mathbb{S}(\nabla \bu):\nabla A[w] - \delta \int_{\mathscr{O}_\varepsilon([\tilde{w}]_{cor})} \nabla(\rho \tilde{\bu})\cdot \nabla A[w] - \delta  \int_{\mathscr{O}_\varepsilon([\tilde{w}]_{cor})} \rho \tilde{\bu}\cdot A[w]\nonumber\\
&&\leq C(K, \delta^{-1},\bb{u}_B ) \|A[w]\|_{H^1(\mathscr{O}_\varepsilon([\tilde{w}]_{cor}))} \nonumber\\
&&\leq C(\kappa^{-1}, K, \delta^{-1}, \bu_B)+ \frac\kappa2 \int_\Gamma |\Delta w|^2, \label{wH2}
\end{eqnarray}
where we used $\eqref{uH1}$ and the imbedding
\begin{eqnarray*}
&&\|\rho\|_{L^\infty( \mathscr{O}_\varepsilon([\tilde{w}]_{cor}))} \leq C\|\rho\|_{W^{1,6}( \mathscr{O}_\varepsilon([\tilde{w}]_{cor}))}\leq C\|\rho\|_{H^{2}( \mathscr{O}_\varepsilon([\tilde{w}]_{cor}))} \leq C(K,\delta^{-1},\bu_B),
\end{eqnarray*}
which follows by $\eqref{H2:rho}$. Here, we have also use that the corrected domain $\mathscr{O}_\varepsilon([\tilde{w}]_{cor}))$ has a uniform imbedding constant since it is uniformly Lipschitz. Now,
one can estimate for any $\psi \in C_c^\infty(\Gamma)$
\begin{eqnarray}
&&\int_\Gamma \kappa\Delta w \Delta \psi\nonumber\\
&&\leq\int_{\mathscr{O}_\varepsilon([\tilde{w}]_{cor})} (T(\rho) \tilde{\bu}\otimes \tilde{\bu}): \nabla A[\psi] +\int_{\mathscr{O}_\varepsilon([\tilde{w}]_{cor})} p_{\varepsilon,\delta}(\rho) \divg  A[\psi] \nonumber \\
&&\quad -\int_{\mathscr{O}_\varepsilon([\tilde{w}]_{cor})}\mathbb{S}(\nabla \bu):\nabla A[\psi] - \delta \int_{\mathscr{O}_\varepsilon([\tilde{w}]_{cor})} \nabla(\rho \tilde{\bu})\cdot \nabla A[\psi] - \delta  \int_{\mathscr{O}_\varepsilon([\tilde{w}]_{cor})} \rho \tilde{\bu}\cdot A[\psi]\nonumber\\
&&\leq C(K, \delta^{-1}, \bb{u}_B) \|A[\psi]\|_{H^1(\mathscr{O}_\varepsilon([\tilde{w}]_{cor}))} \nonumber\\
&&\leq C(K, \delta^{-1}, \bb{u}_B)\|\psi\|_{H^{\frac12}(\Gamma)} \label{wH72}
\end{eqnarray}
so the elliptic regularity implies \cite{grisvard,karmanplates}
\begin{eqnarray}
\kappa\|w\|_{H^{\frac72}(\Gamma)} \leq C( K,\delta^{-1},\bb{u}_B) \label{72est}
\end{eqnarray}
so $\mathcal{T}$ is compact in $w$, which combined with $\eqref{u:w232}$ gives us the compactness in $\bb{u}$ by Lemma \ref{conv:lemma}.

\bigskip
\noindent
\textbf{Continuity}: Let $(\bb{u}_n,w_n)\to (\bb{u},w)$ in $V$. The goal is to show that
\begin{eqnarray*}
(\bb{U}_n,\bb{W}_n):=\mathcal{T}(\bb{u}_n,w_n) \to \mathcal{T}(\bb{u},w)=: (\bb{U},W) ~ \text{ in } ~ V.
\end{eqnarray*}
Due to compactness, we know that $(\bb{U}_n,\bb{W}_n)$ converges strongly in $V$, so it is enough to show that $(\bb{U}_n,\bb{W}_n) \rightharpoonup (\bb{U},W)$ in $V$, since the weak and strong limits coincide if they exist. First, denote $\rho_n = \rho_n[\bb{u}_n,w_n]$ which is the solution to the elliptic problem
\begin{eqnarray}
&&\delta\int_{\mathscr{O}([w_n]_{cor})} \nabla\rho_n \cdot \nabla \varphi +  \delta\int_{\mathscr{O}([w_n]_{cor})} \rho_n \varphi  - \int_{\mathscr{O}([w_n]_{cor})} T(\rho_n)\bb{u}_n \cdot \nabla \varphi \nonumber \\
&&= -\int_{\Sigma_{in}} \rho_B^\delta\bu_B\cdot \bb{n} \varphi - \int_{\Sigma_{out}} T(\rho_n)\bu_B\cdot \bb{n} \varphi \label{app:fixed:cont:alt}
\end{eqnarray}
that holds for all $\varphi \in C^\infty(\overline{\mathscr{O}([w_n]_{cor})})$. By Lemma \ref{conv:lemma}, one has
\begin{eqnarray*}
&&\rho_n \chi_{|\mathscr{O}([w_n]_{cor})} \to \rho \chi_{|\mathscr{O}_\varepsilon([w]_{cor})} \quad \text{in} \quad L^2(\mathbb{R}^3),\\
&&\nabla \rho_n \chi_{|\mathscr{O}([w_n]_{cor})} \to \nabla \rho \chi_{|\mathscr{O}_\varepsilon([w]_{cor})} \quad \text{in} \quad L^2(\mathbb{R}^3),
\end{eqnarray*}
while strong convergence of $(\bb{u}_n,w_n)$ gives us 
\begin{eqnarray*}
\lim_{n\to \infty}\| \bb{u} - \bb{u}_n\|_{H^1(\mathscr{O}_{max})}=0,
\end{eqnarray*}
so we can pass to the limit $n\to \infty$ in $\eqref{app:fixed:cont:alt}$ to obtain 
\begin{eqnarray*}
&&\delta\int_{\mathscr{O}_\varepsilon([w]_{cor})} \nabla\rho \cdot \nabla \varphi +  \delta\int_{\mathscr{O}_\varepsilon([w]_{cor})} \rho \varphi  - \int_{\mathscr{O}_\varepsilon([w]_{cor})} T(\rho)\bb{u} \cdot \nabla \varphi \nonumber \\
&&= -\int_{\Sigma_{in}} \rho_B^\delta\bu_B\cdot \bb{n} \varphi - \int_{\Sigma_{out}} T(\rho)\bu_B\cdot \bb{n} \varphi 
\end{eqnarray*}
which implies $\rho = \rho[\bb{u},w]$ by the uniqueness of solution of $\eqref{app:fixed:cont}$. Noticing that, by imbedding of Sobolev spaces and Lemma \ref{conv:lemma}
\begin{eqnarray*}
&&\bb{u}_n \chi_{|\mathscr{O}([w_n]_{cor})} \to \bb{u} \chi_{|\mathscr{O}_\varepsilon([w]_{cor})} \quad \text{in} \quad L^p(\mathbb{R}^3) , \quad p<6, \\
&&\nabla\rho_n \chi_{|\mathscr{O}([w_n]_{cor})} \to \nabla \rho \chi_{|\mathscr{O}_\varepsilon([w]_{cor})} \quad \text{in} \quad L^p(\mathbb{R}^3) , \quad p<6,
\end{eqnarray*}
and observing that $\bb{U}_n$ satisfies
\begin{eqnarray}
&&\int_{\mathscr{O}([w_n]_{cor})}\mathbb{S}(\nabla \bb{U}_n):\nabla\boldsymbol{\varphi}+\delta  \int_{\mathscr{O}([w_n]_{cor})} \rho_n\nabla \bb{U}_n\cdot \nabla\boldsymbol{\varphi} + \delta  \int_{\mathscr{O}([w_n]_{cor})} \rho_n \bb{U}_n\cdot\boldsymbol{\varphi} \nonumber \\
&&=-\delta\int_{\mathscr{O}([w_n]_{cor})}  \bb{u}_n\otimes\nabla\rho_n: \nabla\boldsymbol{\varphi} +\int_{\mathscr{O}([w_n]_{cor})} (T(\rho_n) \bb{u}_n\otimes \bb{u}_n): \nabla\boldsymbol{\varphi} + \int_{\mathscr{O}([w_n]_{cor})} p_{\varepsilon,\delta}(\rho_n) \divg  \boldsymbol{\varphi} \nonumber \\ &&\label{app:fixed:mom:alt}
\end{eqnarray}
for every $\boldsymbol\varphi\in C_c^\infty(\mathscr{O}([w_n]_{cor}))$, we can pass to the limit $n\to \infty$ in $\eqref{app:fixed:mom:alt}$ to conclude that
\begin{eqnarray}
&&\int_{\mathscr{O}_\varepsilon([w]_{cor})}\mathbb{S}(\nabla \overline{\bb{U}}):\nabla\boldsymbol{\varphi}+\delta  \int_{\mathscr{O}_\varepsilon([w]_{cor})} \rho\nabla \overline{\bb{U}}\cdot \nabla\boldsymbol{\varphi} + \delta  \int_{\mathscr{O}_\varepsilon([w]_{cor})} \rho \overline{\bb{U}}\cdot\boldsymbol{\varphi} \nonumber \\
&&=-\delta\int_{\mathscr{O}_\varepsilon([w]_{cor})}  \bb{u}\otimes\nabla\rho: \nabla\boldsymbol{\varphi} +\int_{\mathscr{O}_\varepsilon([w]_{cor})} (T(\rho) \bb{u}\otimes \bb{u}): \nabla\boldsymbol{\varphi} + \int_{\mathscr{O}_\varepsilon([w]_{cor})} p_{\varepsilon,\delta}(\rho) \divg  \boldsymbol{\varphi} \label{app:fixed:mom:alt2}
\end{eqnarray}
for every $\boldsymbol\varphi\in C_c^\infty(\mathscr{O}_\varepsilon([w]_{cor}))$, where $\overline{\bb{U}}$ is the weak limit of $\bb{U}$ in the following sense
\begin{eqnarray*}
&&\bb{U}_n \chi_{|\mathscr{O}([w_n]_{cor})} \rightharpoonup \overline{\bb{U}} \chi_{|\mathscr{O}_\varepsilon([w]_{cor})} \quad \text{in} \quad L^2(\mathbb{R}^3),\\
&&\nabla\bb{U}_n \chi_{|\mathscr{O}([w_n]_{cor})} \rightharpoonup \nabla \overline{\bb{U}} \chi_{|\mathscr{O}_\varepsilon([w]_{cor})} \quad \text{in} \quad L^2(\mathbb{R}^3).
\end{eqnarray*}
Since the solution to $\eqref{app:fixed:mom:alt2}$ is unique, we conclude that $\overline{\bb{U}}=\bb{U} = \bb{U}[\rho[\bb{u},w],\bb{u},w]$, which directly implies that $W = W[\rho,\bb{u},w]$. Therefore, $\mathcal{T}$ is continuous.

\bigskip

\noindent
\textbf{Uniform estimates of the fixed-points set}: The last step is to show that the set 
$$\{(\bu,w) = \theta \mathcal{T}(\bu,w), \theta \in [0,1]\}$$
is bounded. Therefore, assume that $\theta\in [0,1]$ and let $(\bu,w) = \theta \mathcal{T}(\bu,w)$. First, let us point out that testing continuity equation with $1$ gives us
\begin{eqnarray}
\delta \int_{ \mathscr{O}_\varepsilon([w]_{cor})} \rho +\int_{\Sigma_{out}} \underbrace{ T(\rho)\bu_B\cdot \bb{n}}_{\geq 0}= -\int_{\Sigma_{in}} \rho_B^\delta\bu_B\cdot \bb{n} \leq C(\bb{u}_B). \label{rhoL1}
\end{eqnarray}
Then from the fluid momentum equation we obtain
\begin{eqnarray}
&& c\|\bb{u}-\bb{u}_B\|_{H^1( \mathscr{O}_\varepsilon([w]_{cor}))}^2 \nonumber \\
&&\leq \int_{ \mathscr{O}_\varepsilon([w]_{cor})}\mathbb{S}(\nabla \bu-\nabla \bu_B) :(\nabla\bu-\nabla\bu_B)+ \theta  \int_{ \mathscr{O}_\varepsilon([w]_{cor})} \underbrace{p_{\varepsilon,\delta}(\rho) \divg \bb{u}_B}_{\geq 0} \nonumber\\
&&+ \delta (1-\theta) \int_{ \mathscr{O}_\varepsilon([w]_{cor})} \rho \big(|\nabla \bb{u} - \nabla \bb{u}_B|^2 + |\bb{u} - \bb{u}_B|^2\big) \nonumber \\
&&=- \int_{ \mathscr{O}_\varepsilon([w]_{cor})}\mathbb{S}( \nabla \bu_B):(\nabla \bu-\nabla \bb{u}_B)+\theta\int_{ \mathscr{O}_\varepsilon([w]_{cor})} (T(\rho) \bu\otimes \bu): (\nabla\bu-\nabla\bu_B)   \nonumber\\
&&\quad +\theta\int_{ \mathscr{O}_\varepsilon([w]_{cor})} p_{\varepsilon,\delta}(\rho) \divg \bb{u}- \delta\theta  \int_{ \mathscr{O}_\varepsilon([w]_{cor})} \nabla(\rho \bu): (\nabla\bu - \nabla \bu_B) - \delta\theta  \int_{ \mathscr{O}_\varepsilon([w]_{cor})} \rho \bu\cdot (\bu - \bu_B) \nonumber \\
&& \quad -\delta (1-\theta) \int_{ \mathscr{O}_\varepsilon([w]_{cor})} \rho \big( \nabla\bb{u}_B \cdot(\nabla \bb{u} - \nabla \bb{u}_B) + \bb{u}_B\cdot(\bb{u} - \bb{u}_B)\big). \label{uni:Schauder}
\end{eqnarray}
Before we proceed, by using $\eqref{rhoL1}$ and the boundedness $\|\bb{u}_B\|_{W^{1,\infty}(\mathscr{O}_{max})}\leq C$ coming from Lemma \ref{ext:op}, one has
\begin{eqnarray*}
&&\delta (1-\theta) \int_{ \mathscr{O}_\varepsilon([w]_{cor})} \rho \big( \nabla\bb{u}_B \cdot(\nabla \bb{u} - \nabla \bb{u}_B) + \bb{u}_B\cdot(\bb{u} - \bb{u}_B)\big) \\
&&\leq \frac12\underbrace{\delta \int_{ \mathscr{O}_\varepsilon([w]_{cor})} \rho (|\bb{u}_B|^2+ |\nabla \bb{u}_B|^2)}_{\leq C(\bb{u}_B)} + \frac12\delta \underbrace{(1-\theta)^2}_{\leq (1-\theta)} \int_{ \mathscr{O}_\varepsilon([w]_{cor})} \rho \big(|\nabla \bb{u} - \nabla \bb{u}_B|^2 + |\bb{u} - \bb{u}_B|^2\big)  
\end{eqnarray*}
so the second term on the right-hand side can be absorbed. In order to control the remaining terms on the right-hand side of $\eqref{uni:Schauder}$, we can first renormalize the continuity equation (in the strong form) by multiplying it with $G'(\rho)$ to obtain
\begin{eqnarray*}
&&\delta G''(\rho)|\nabla \rho|^2+\delta G'(\rho)\rho -\delta \divg (G'(\rho)\nabla\rho)+  \theta\divg(H(\rho)\bu) + \theta[G'(\rho)T(\rho)-H(\rho)]\divg \bu = 0
\end{eqnarray*}
where $H'(\rho)=G'(\rho)T(\rho)$, so by choosing $G = G_{\varepsilon,\delta}$ such that
\begin{eqnarray*}
G_{\varepsilon,\delta}''(\rho) = \frac{p_{\varepsilon,\delta}'(\rho)}{T(\rho)} \quad \text{ or equivalently } \quad G_{\varepsilon,\delta}'(\rho)T(\rho) - H_{\varepsilon,\delta}(\rho) = p_{\varepsilon,\delta}(\rho),
\end{eqnarray*}
we can express the term $\theta p_{\varepsilon,\delta}(\rho) \divg \bu$ in $\eqref{uni:Schauder}$. This allows us to repeat the calculation from \cite[Section 5]{FN2018} to obtain
\begin{eqnarray}
&&\| \bu \|_{H^1( \mathscr{O}_\varepsilon([w]_{cor}))}^2 + \delta \theta\int_{ \mathscr{O}_\varepsilon([w]_{cor})} \big[ G_{\varepsilon,\delta}''(\rho)|\nabla \rho|^2 + \rho G_{\varepsilon,\delta}'(\rho) \big] \leq c\| \bu_B\|_{H^1( \mathscr{O}_{max})}^2 \label{u:est}
\end{eqnarray}
where $c$ is uniform with respect to $\delta$ and $\varepsilon$, so estimating as in $\eqref{H1:rho},\eqref{H2:rho}$, one has
\begin{eqnarray*}
\|\rho\|_{H^2( \mathscr{O}_\varepsilon([w]_{cor}))} \leq C(\overline{
	\rho},\delta^{-1}, \bb{u}_B). \label{rho:est}
	\end{eqnarray*}
	With these estimates, one can repeat the calculations from $\eqref{wH2}$ and $\eqref{72est}$ to conclude
	\begin{eqnarray*}
\kappa\|w\|_{H^{\frac72}(\Gamma)} \leq C( \overline{\rho},\delta^{-1},\bb{u}_B).
\end{eqnarray*}
This concludes the proof of boundedness and consequently ensures existence of a fixed point $\mathcal{T}(\bb{u},w) = (\bb{u},w)  \in H_B^1(\mathscr{O}_{max})\times \big[H_0^2(\Gamma)\cap H^{3+a}(\Gamma)\big]$  by the Schaefer corollary to the Schauder Theorem, which by construction satisfies $(\bb{u},w) \in V$. Therefore, we have shown that there exists a weak solution $(\rho,\bb{u},w)$ in the sense of Definition $\ref{app:sol}$.

\subsection{Limit $\varepsilon\to 0$}\label{sec:limit:eps}
Denote $(\rho_\varepsilon,\bb{u}_\varepsilon,w_\varepsilon)$ the solution in the sense of Definition $\ref{app:sol}$ obtained in the previous section. First, noticing that
\begin{eqnarray*}
G_{\varepsilon,\delta}''(\rho_\varepsilon)= \frac{p_{\varepsilon,\delta}'(\rho_\varepsilon)}{T(\rho)} = \frac{p_\varepsilon'(\rho_\varepsilon)+\sqrt{\delta}\rho_\varepsilon }{T(\rho_\varepsilon)} \geq \frac{\sqrt{\delta}}{\overline{\rho}}
\end{eqnarray*}
we can control
\begin{eqnarray*}
\frac{\delta^{\frac32}}{\overline{\rho}}\int_{\mathscr{O}_\varepsilon([w_\varepsilon]_{cor})}|\nabla \rho_\varepsilon|^2 \leq    \int_{\mathscr{O}_\varepsilon([w_\varepsilon]_{cor})} \delta G_{\varepsilon,\delta}''(\rho_\varepsilon)|\nabla \rho_\varepsilon|^2 \label{h1:rho:eps*}
\end{eqnarray*}
so the estimate $\eqref{u:est}$ for $\theta =1$ give us the uniform bound
\begin{eqnarray}
&&\| \bu_\varepsilon \|_{H^1(\mathscr{O}_\varepsilon([w]_{cor})))}^2 + \frac{\delta^{\frac32}}{\overline{\rho}}\int_{\mathscr{O}_\varepsilon([w_\varepsilon]_{cor})}|\nabla \rho_\varepsilon|^2  \leq c\| \bu_B\|_{H^1(\mathscr{O}_{max})}^2 \label{u:rho:eps}
\end{eqnarray}
which, by estimating as in $\eqref{H2:rho}$, implies
\begin{eqnarray}
&&\|\rho_\varepsilon\|_{W^{1,6}(\mathscr{O}_\varepsilon([w_\varepsilon]_{cor}))}\leq C\|\rho_\varepsilon\|_{H^2(\mathscr{O}_\varepsilon([w_\varepsilon]_{cor}))} \leq C(\delta^{-1},\bb{u}_B),\label{rho:h2:eps}
\end{eqnarray}
while then estimating as in $\eqref{H72:delta}$ gives us
\begin{eqnarray*}
\kappa\|w_\varepsilon\|_{H^{\frac72}(\Gamma)} \leq C(\delta^{-1},\bb{u}_B). 
\end{eqnarray*}
We thus conclude
\begin{eqnarray*}
\bu_\varepsilon\chi_{|\mathscr{O}_\varepsilon([w_\varepsilon]_{cor})} \rightharpoonup \bu\chi_{|\mathscr{O}([w]_{cor})},& \quad &\text{ weakly in } L^2(\mathbb{R}^3), \\
\nabla \bu_\varepsilon\chi_{|\mathscr{O}_\varepsilon([w_\varepsilon]_{cor})} \rightharpoonup \nabla\bu\chi_{|\mathscr{O}([w]_{cor})},& \quad &\text{ weakly in } L^2(\mathbb{R}^3), \\
\rho_\varepsilon\chi_{|\mathscr{O}_\varepsilon([w_\varepsilon]_{cor})} \rightharpoonup \rho\chi_{|\mathscr{O}([w]_{cor})},& \quad &\text{ weakly in } L^2(\mathbb{R}^3), \\
\nabla \rho_\varepsilon\chi_{|\mathscr{O}_\varepsilon([w_\varepsilon]_{cor})} \rightharpoonup \nabla\rho\chi_{|\mathscr{O}([w]_{cor})},& \quad &\text{ weakly in } L^2(\mathbb{R}^3), \\
w_\varepsilon  \rightharpoonup w, & \quad &\text{ weakly in } H_0^2(\Gamma)\cap H^{\frac72}(\Gamma),
\end{eqnarray*}
as $\varepsilon\to 0$, at least for a subsequence. Next, following the steps from \cite[Section 6]{FN2018}, one obtains
$$   0\leq \rho < \overline{\rho} \quad \text{ a.e. on }  \mathscr{O}([w]_{cor}), $$
and
\begin{eqnarray*}
\|p_{\varepsilon,\delta}(\rho_\varepsilon)\|_{L^2( \mathscr{O}_\varepsilon([w_\varepsilon]_{cor}))} \leq C(\delta^{-1}), 
\end{eqnarray*}
so by the estimate $\eqref{rho:h2:eps}$ and Lemma \ref{conv:lemma}
\begin{eqnarray*}
&p_{\varepsilon,\delta}(\rho_\varepsilon)\chi_{|\mathscr{O}_\varepsilon([w_\varepsilon]_{cor})}  \to  p_{\delta}(\rho)\chi_{|\mathscr{O}([w]_{cor})},& \quad \text{ in } L^2(\mathbb{R}^3).
\end{eqnarray*}
Therefore, one obtains that the limiting functions satisfy
\begin{eqnarray}
&&\int_{ \mathscr{O}([w]_{cor})} (\rho \bu\otimes \bu): \nabla\boldsymbol{\varphi} + \int_{  \mathscr{O}([w]_{cor})} (p(\rho)+\sqrt{\delta}\rho) \divg  \boldsymbol{\varphi}- \int_{  \mathscr{O}([w]_{cor})}\mathbb{S}(\nabla \bu):\nabla\boldsymbol{\varphi} \nonumber \\
&&- \delta  \int_{  \mathscr{O}([w]_{cor})} \nabla(\rho \bu): \nabla\boldsymbol{\varphi} - \delta  \int_{  \mathscr{O}([w]_{cor})} \rho \bu\cdot\boldsymbol{\varphi}  -\int_\Gamma \kappa\Delta w \Delta \psi= 0 \label{mom:eps}
\end{eqnarray}
for all $\boldsymbol{\varphi} \in C_c^\infty( \mathscr{O}([w]_{cor})\cup \Gamma^{[w]_{cor}})$ and $\psi \in C_c^\infty(\Gamma)$ such that $\boldsymbol{\varphi}_{|\Gamma^{[w]_{cor}}}  = \psi \bb{e}_z$, and
\begin{eqnarray*}
&&\delta\int_{ \mathscr{O}([w]_{cor})} \nabla\rho \cdot \nabla \varphi +  \delta\int_{ \mathscr{O}([w]_{cor})} \rho \varphi  - \int_{ \mathscr{O}([w]_{cor})} \rho\bb{u} \cdot \nabla \varphi\\
&&= -\int_{\Sigma_{in}} \rho_B^\delta\bu_B\cdot \bb{n} \varphi - \int_{\Sigma_{out}} \rho\bu_B\cdot \bb{n} \varphi \label{cont:eps}
\end{eqnarray*}
holds for all $\varphi \in C^\infty(\overline{ \mathscr{O}([w]_{cor})})$.

\subsection{Limit $\delta\to 0$}\label{sec:limit:delta}
Next, denote the solution obtained in previous section as $(\rho_\delta,\bb{u}_\delta, w_\delta)$ which solves $\eqref{mom:eps}$ and $\eqref{cont:eps}$. Since $\|[w_\delta]_{cor}\|_{C^{0,1}(\Gamma)}\leq \frac14$ and $[w_\delta]_{cor}\in C_0(\Gamma)$, there exists $s\in C_0(\Gamma)$ and a converging subsequence
$$[w_\delta]_{cor} \to s \quad \text{ in } C(\Gamma)$$
as $\delta \to 0$. Moreover, from $\eqref{u:rho:eps}$, the following estimates hold
\begin{eqnarray*}
\| \bu_\delta \|_{H^1(\mathscr{O}([w]_{cor})))}^2+\frac{\delta^{\frac32}}{\overline{\rho}}\int_{\mathscr{O}([w]_{cor})}|\nabla \rho_{\delta}|^2 \leq c\| \bu_B\|_{H^1(\mathscr{O}_{max})}^2.\label{h1:rho:eps}
\end{eqnarray*}
Finally, since $\rho_\delta < \overline{\rho}$, we can combine the above estimates to conclude
\begin{eqnarray*}
\bu_\delta\chi_{|\mathscr{O}([w_\delta]_{cor})} \rightharpoonup \bu\chi_{|\mathscr{O}(s)},& \quad &\text{ weakly in } L^2(\mathbb{R}^3), \\
\nabla \bu_\delta\chi_{|\mathscr{O}([w_\delta]_{cor})} \rightharpoonup \nabla\bu\chi_{|\mathscr{O}(s)},& \quad &\text{ weakly in } L^2(\mathbb{R}^3), \\
\rho_\delta\chi_{|\mathscr{O}([w_\delta]_{cor})} \rightharpoonup \rho\chi_{|\mathscr{O}(s)},& \quad &\text{ weakly in } L^2(\mathbb{R}^3), \\
\delta\nabla \rho_\delta\chi_{|\mathscr{O}_\delta([w_\delta]_{cor})} \rightharpoonup 0,& \quad &\text{ weakly in } L^2(\mathbb{R}^3),
\end{eqnarray*}
as $\delta \to 0$. Since by Lemma \ref{conv:lemma}
$$ \bu_\delta\chi_{|\mathscr{O}([w_\delta]_{cor})} \to \bu\chi_{|\mathscr{O}(s)}, \quad \text{ in } L^2(\mathbb{R}^3) $$
we can pass to the limit in the continuity equation to conclude that
\begin{eqnarray*}
\int_{ \mathscr{O}(s)}\rho\bb{u} \cdot \nabla \varphi = \int_{\Sigma_{in}} \rho_B\bu_B\cdot \bb{n} \varphi 
\end{eqnarray*}
holds for all $\varphi \in C^\infty(\overline{ \mathscr{O}(s)})$ such that $\varphi_{|\Sigma_{out}} = 0$. Since $\rho_B<\overline{\rho}$, we cannot have $\rho = \overline{\rho}$ a.e. on $\mathscr{O}(s)$. Indeed, assuming that $\rho = \overline{\rho}$ a.e. on $\mathscr{O}(s)$,  it is easy to construct a sequence of test functions $\varphi_n \in C^\infty(\overline{ \mathscr{O}(s)})$ such that
\begin{eqnarray*}
{\varphi_n}_{|\Sigma_{out}} = 0,  \quad {\varphi_n}_{|\Sigma_{left}} = \psi \in C_c^\infty(\Sigma_{left})\text{ with } {\psi}_{|\Sigma_{in}}=1,
\end{eqnarray*}
and
$$\lim\limits_{n\to \infty} \int_{ \mathscr{O}(s)}\bb{f} \cdot \nabla \varphi_n = \int_{ \Sigma_{left}}\bb{f}\cdot \bb{n}\psi \quad \text{ for any }\bb{f} \in H^1(\mathscr{O}(s)).$$
Plugging $\varphi_n$ into the continuity equation and passing to the limit $n\to \infty$ gives us
\begin{eqnarray*}
\int_{\Sigma_{in}} \rho_B\bu_B\cdot \bb{n}= \lim\limits_{n\to \infty} \int_{ \mathscr{O}(s)}\rho\bb{u} \cdot \nabla \varphi_n = \overline{\rho}\lim\limits_{n\to \infty} \int_{ \mathscr{O}(s)}\bb{u} \cdot \nabla \varphi_n = \overline{\rho}\int_{\Sigma_{in}}\bb{u}_B \cdot \bb{n} 
\end{eqnarray*}
which is obviously contradiction with $\rho_B< \overline{\rho}$. Thus, we can conclude that
$$\frac{1}{|\mathscr{O}(s)|} \int_{\mathscr{O}(s)}\rho < \overline{\rho}$$
and consequently there exists $\beta>1$ such that for a non-relabeled subsequence
$$\limsup_{\delta\to 0}\frac{\beta}{|\mathscr{O}([w_\delta]_{cor})|} \int_{\mathscr{O}([w_\delta]_{cor}))}\rho_\delta<\overline{\rho}. $$
This allows us to bound the pressure as in \cite[Section 7.1]{FN2018} by testing the $\eqref{mom:eps}$ with $B_{ \mathscr{O}([w]_{cor})}\left[\rho - \frac1{|\mathscr{O}([w]_{cor})|}\int_{ \mathscr{O}([w]_{cor})}\rho\right]$ where $B_{ \mathscr{O}([w]_{cor})}: L^2( \mathscr{O}([w]_{cor})) \to H_0^1( \mathscr{O}([w]_{cor}))$ is the inverse divergence Bogovskii operator, and obtain
\begin{eqnarray*}
\int_{ \mathscr{O}([w]_{cor})} \rho p_{\varepsilon,\delta}(\rho) \leq C((\beta-1)^{-1},\| \bu_B\|_{H^1(\mathscr{O}_{max})}).
\end{eqnarray*}
Note that by construction that $\mathscr{O}([w]_{cor})$ is uniformly Lipschitz, which ensures a uniform Bogovskii constant. Moreover, the constant on the right depends on $(\beta-1)^{-1}$ which is a consequence of a contradiction argument and thus seems to be non-quantifiable, so this estimate might be very large and it is not obvious how to reduce it. We continue by testing $\eqref{mom:eps}$ with $B_{ \mathscr{O}([w_\delta]_{cor})}\left[p(\rho_\delta)^\alpha - \frac1{|\mathscr{O}([w_\delta]_{cor})|}\int_{ \mathscr{O}([w_\delta]_{cor})} p(\rho_\delta)^\alpha\right]$ to bound
\begin{eqnarray}
&&\int_{ \mathscr{O}([w_\delta]_{cor})}(p(\rho_\delta)+\sqrt{\delta}\rho_\delta)p(\rho_\delta)^\alpha \nonumber\\
&&=\int_{ \mathscr{O}([w_\delta]_{cor})}(p(\rho_\delta)+\sqrt{\delta}\rho_\delta) \frac1{|\mathscr{O}([w_\delta]_{cor})|}\int_{ \mathscr{O}([w_\delta]_{cor})} p(\rho_\delta)^\alpha \nonumber\\
&&\quad-  \int_{ \mathscr{O}([w_\delta]_{cor})}\rho_\delta \bb{u}_\delta\otimes \bb{u}_\delta:\nabla B_{ \mathscr{O}([w_\delta]_{cor})}\left[p(\rho_\delta)^\alpha -\frac1{|\mathscr{O}([w_\delta]_{cor})|} \int_{ \mathscr{O}([w_\delta]_{cor})} p(\rho_\delta)^\alpha\right]\nonumber \\
&&\quad+\int_{ \mathscr{O}([w_\delta]_{cor})} \mathbb{S}(\nabla \bb{u}_\delta):\nabla B_{ \mathscr{O}([w_\delta]_{cor})}\left[p(\rho_\delta)^\alpha -\frac1{|\mathscr{O}([w_\delta]_{cor})|} \int_{ \mathscr{O}([w_\delta]_{cor})} p(\rho_\delta)^\alpha\right] \nonumber\\
&&\quad+ \delta  \int_{ \mathscr{O}([w_\delta]_{cor})}\nabla(\rho_\delta \bb{u}_\delta):\nabla B_{ \mathscr{O}([w_\delta]_{cor})}\left[p(\rho_\delta)^\alpha -\frac1{|\mathscr{O}([w_\delta]_{cor})|} \int_{ \mathscr{O}([w_\delta]_{cor})} p(\rho_\delta)^\alpha\right]\nonumber \\
&&\quad+\delta  \int_{ \mathscr{O}([w_\delta]_{cor})}\rho_\delta \bb{u}_\delta\cdot  B_{ \mathscr{O}([w_\delta]_{cor})}\left[p(\rho_\delta)^\alpha -\frac1{|\mathscr{O}([w_\delta]_{cor})|} \int_{ \mathscr{O}([w_\delta]_{cor})} p(\rho_\delta)^\alpha\right] \nonumber\\
&&\leq C + \overline{\rho}\|\bb{u}_\delta\|_{L^6(\mathscr{O}([w_\delta]_{cor}))}^2 \left(\|p(\rho_\delta)^\alpha\|_{L^{\frac32}(\mathscr{O}([w_\delta]_{cor}))}+C\right)+ \|\bb{u}_\delta\|_{H^1(\mathscr{O}([w_\delta]_{cor}))} \left(\|p(\rho_\delta)^\alpha\|_{L^2(\mathscr{O}([w_\delta]_{cor}))}+C\right) \nonumber\\
&&\quad+\delta\left( \|\nabla \rho_\delta \|_{L^2(\mathscr{O}([w_\delta]_{cor}))}\|\bb{u}_\delta \|_{L^{6}(\mathscr{O}([w_\delta]_{cor}))}+ C\overline{\rho}\|\bb{u}_\delta \|_{H^1(\mathscr{O}([w_\delta]_{cor}))} \right)        \left(\|p(\rho_\delta)^\alpha\|_{L^3(\mathscr{O}([w_\delta]_{cor}))}+C\right)\nonumber\\
&&\quad + \delta \overline{\rho}\|\bb{u}_\delta \|_{L^6(\mathscr{O}([w_\delta]_{cor}))} \left(\|p(\rho_\delta)^\alpha\|_{L^{\frac65}(\mathscr{O}([w_\delta]_{cor}))}+C\right) \nonumber\\
&&\leq C + \frac12\|p(\rho_\delta)^\alpha\|_{L^3(\mathscr{O}([w_\delta]_{cor}))}^3 = C + \frac12\int_{ \mathscr{O}([w_\delta]_{cor})} p(\rho_\delta)^{3\alpha} \label{press:delta}
\end{eqnarray}
for $C=C((\beta-1)^{-1},\| \bu_B\|_{H^1(\mathscr{O}_{max})})$, where the last term can be absorbed for $\alpha = \frac12$, so
\begin{eqnarray}
\int_{ \mathscr{O}([w_\delta]_{cor})}(p(\rho_\delta)+\sqrt{\delta}\rho_\delta) p(\rho_\delta)^{\frac12} \leq  C(\| \bu_B\|_{H^1(\mathscr{O}_{max})}). \label{press:est}
\end{eqnarray}
Finally, by choosing $\psi= w_\delta$ and $\boldsymbol{\varphi} = A[w_\delta]$ in $\eqref{mom:eps}$, where $A$ is the extension operator from Lemma $\ref{extension}$, one obtains

\begin{eqnarray*}
&&\int_\Gamma \kappa|\Delta w_\delta|^2  \nonumber\\
&&= \int_{ \mathscr{O}([w_\delta]_{cor})} (\rho_\delta \bu_\delta\otimes \bu_\delta): \nabla A[w_\delta] + \int_{  \mathscr{O}([w_\delta]_{cor})} (p(\rho_\delta)+\sqrt{\delta}\rho_\delta) \divg  A[w_\delta]- \int_{  \mathscr{O}([w_\delta]_{cor})}\mathbb{S}(\nabla \bu_\delta):\nabla A[w_\delta] \nonumber \\
&&\quad - \delta  \int_{  \mathscr{O}([w_\delta]_{cor})} \nabla(\rho_\delta \bu_\delta): \nabla A[w_\delta] - \delta  \int_{  \mathscr{O}([w_\delta]_{cor})} \rho_\delta \bu_\delta\cdot A[w_\delta] \nonumber\\ 
&&\leq C \overline{\rho} \|\bb{u}_\delta\|_{L^6( \mathscr{O}([w_\delta]_{cor}))}^2 \|\nabla A[w_\delta]\|_{L^2( \mathscr{O}([w_\delta]_{cor}))}+ \big(\|p(\rho_\delta)\|_{L^{\frac32}( \mathscr{O}([w_\delta]_{cor}))}+C\delta \overline{\rho} \big) \|\nabla A[w_\delta]\|_{L^3( \mathscr{O}([w_\delta]_{cor}))}\nonumber \\
&&\quad + \|\nabla \bb{u}_\delta\|_{L^2( \mathscr{O}([w_\delta]_{cor}))} \|\nabla A[w_\delta]\|_{L^2( \mathscr{O}([w_\delta]_{cor}))} \nonumber\\
&&\quad +  C\delta   \|\nabla \rho_\delta\|_{L^2( \mathscr{O}([w_\delta]_{cor}))}\| \bb{u}_\delta\|_{L^6( \mathscr{O}([w_\delta]_{cor}))}\|\nabla A[w_\delta]\|_{L^3( \mathscr{O}([w_\delta]_{cor}))} + \overline{\rho} \| \nabla \bb{u}_\delta\|_{L^2( \mathscr{O}([w_\delta]_{cor}))} \|\nabla A[w_\delta]\|_{L^2( \mathscr{O}([w_\delta]_{cor}))}\nonumber \\
&&\quad + \delta\overline{\rho} \|\bb{u}_\delta\|_{L^2( \mathscr{O}([w_\delta]_{cor}))} \|A[w_\delta]\|_{L^2( \mathscr{O}([w_\delta]_{cor}))}\nonumber \\
&&\leq  C((\beta-1)^{-1},\| \bu_B\|_{H^1(\mathscr{O}_{max})} )\|A[w_\delta]\|_{W^{1,3}( \mathscr{O}([w_\delta]_{cor}))} \nonumber\\
&&\leq \frac1\kappa C((\beta-1)^{-1},\| \bu_B\|_{H^1(\mathscr{O}_{max})} )+ \frac\kappa2 \int_\Gamma |\Delta w_\delta|^2, 
\end{eqnarray*}
implying
$$ \kappa\|w_\delta\|_{H^2(\Gamma)}\leq C((\beta-1)^{-1},\| \bu_B\|_{H^1(\mathscr{O}_{max})} ), $$
while using the estimate $\eqref{press:est}$,  for any $\psi\in C_c^\infty(\Gamma)$ one has
\begin{eqnarray*}
&&\int_\Gamma \kappa\Delta w_\delta \Delta \psi \nonumber\\
&&=-\int_{ \mathscr{O}([w_\delta]_{cor})}\mathbb{S}(\nabla \bu_\delta):\nabla A[\psi]+\int_{ \mathscr{O}([w_\delta]_{cor})} (\rho_\delta \bu_\delta\otimes \bu_\delta): \nabla A[\psi]  \nonumber \\
&&\quad + \int_{ \mathscr{O}([w_\delta]_{cor})} (p(\rho_\delta)+\sqrt{\delta}\rho_\delta) \divg  A[\psi] - \delta  \int_{ \mathscr{O}([w_\delta]_{cor})} \nabla(\rho_\delta \bu_\delta):\nabla A[\psi] - \delta \int_{ \mathscr{O}([w_\delta]_{cor})} \rho_\delta \bu_\delta: A[\psi] \nonumber\\
&&\leq C( (\beta-1)^{-1},\| \bu_B\|_{H^1(\mathscr{O}_{max})})\left( \| A[\psi]\|_{W^{1,3}( \mathscr{O}([w]_{cor}))}+\| \psi\|_{H^{\frac12}(\Gamma)} \right) \nonumber\\
&&\leq C((\beta-1)^{-1}, \| \bu_B\|_{H^1(\mathscr{O}_{max})}) \| \psi\|_{W^{\frac23,3}(\Gamma)}, \label{H72:delta}
\end{eqnarray*}
so by classic elliptic regularity for the biharmonic operator \cite{karmanplates,grisvard}
\begin{eqnarray*}
\|w_\delta\|_{W^{\frac{10}3,3}(\Gamma)} \leq C(\kappa^{-1},(\beta-1)^{-1},\| \bu_B\|_{H^1(\mathscr{O}_{max})})
\end{eqnarray*}
which implies
$$ w_\delta  \rightharpoonup w,  \quad \text{ weakly in } H_0^2(\Gamma)\cap W^{\frac{10}3,3}(\Gamma). $$
Note that the compact imbedding of $W^{\frac{10}3,3}(\Gamma)$ into $C^{0,1}(\Gamma)$ implies that 
$$[w_\delta]_{cor} \to [w]_{cor} = s \quad \text{ in } C^{0,1}(\Gamma).$$
We now have the following estimates independent of $\delta$
\begin{eqnarray*}
&&\| \bu_\delta \|_{L^6(\mathscr{O}([w_\delta]_{cor}))}^2\leq C\| \bu_\delta \|_{H^1(\mathscr{O}([w_\delta]_{cor}))}^2 \leq  c\| \bu_B\|_{H^1(\mathscr{O}_{max})}^2,\\
&& 0\leq \rho_\delta <\overline{\rho}, \\
&& \|p(\rho_\delta)\|_{L^{\frac32}(\mathscr{O}([w_\delta]_{cor}))}\leq C((\beta-1)^{-1},\| \bu_B\|_{H^1(\mathscr{O}_{max})}), \\
&&\delta^{\frac34}\|\rho_\delta\|_{H^1(\mathscr{O}([w_\delta]_{cor}))} \leq C((\beta-1)^{-1}, \| \bu_B\|_{H^1(\mathscr{O}_{max})}), \\
&&\|w_\delta\|_{W^{\frac{10}3,3}(\Gamma)} \leq C((\beta-1)^{-1},\| \bu_B\|_{H^1(\mathscr{O}_{max})} ),
\end{eqnarray*}
The strong convergence of density, which ensures the convergence of the pressure, is quite different and more involved than in previous sections. It can be done as by localizing the arguments \cite[Section 7.2]{FN2018}) in the spirit of \cite[Section 6.2]{Breit}, which is based on the arguments developed by Lions \cite{Lions}, so:
\begin{eqnarray*}
(p(\rho_\delta)+\delta \rho_\delta)\chi_{|\mathscr{O}([w_\delta]_{cor})}  \to p(\rho)\chi_{|\mathscr{O}([w]_{cor})}, \quad \text{ in } L^p(\mathbb{R}^3), \quad p<\frac32.
\end{eqnarray*}
Without the $\delta$ terms, we can also repeat the calculation from $\eqref{press:delta}$ to obtain
\begin{eqnarray*}
\|p(\rho)\|_{L^2(\mathscr{O}([w_\delta]_{cor}))}\leq C((\beta-1)^{-1},\| \bu_B\|_{H^1(\mathscr{O}_{max})}).
\end{eqnarray*}
and the calculation from $\eqref{wH2}$ and $\eqref{wH72}$ to obtain
$$\kappa\|w\|_{H^{\frac72}(\Gamma)}\leq C((\beta-1)^{-1},\| \bu_B\|_{H^1(\mathscr{O}_{max})}). $$
We conclude that $(\rho,\bb{u},w)$ is a weak solution on the corrected domain in the sense of Definition \ref{cor:dom:def} and the proof of Theorem \ref{not:main} is finished.

\section{Existence of a weak Solution: Proof of Theorem \ref{main}}
Let $\bb{u}_B,\rho_B$ be given as in Theorem \ref{main}. For every stiffness $\kappa>0$, let $(\rho_\kappa,\bb{u}_\kappa,w_\kappa)$ be a corresponding weak solutions on corrected domains in the sense of Definition \ref{cor:dom:def} constructed in Theorem \ref{not:main} which all satisfy the estimate
$$\| \bu_\kappa \|_{H^1(\mathscr{O}([w_{\kappa}]_{cor})))}  \leq C$$
for $C$ uniform w.r.t. $\kappa$. The goal is to show the existence of a uniform $\beta>1$ w.r.t. $\kappa$ such that 
$$\frac{\beta}{|\mathscr{O}([w_{\kappa}]_{cor})|} \int_{\mathscr{O}([w_{\kappa}]_{cor}))}\rho_{\kappa}<\overline{\rho}. $$
Assume the contrary. Then there exists a convergent sequence $\kappa_n$ such that
$$\lim_{n\to \infty}\frac{1}{|\mathscr{O}([w_{\kappa_n}]_{cor})|} \int_{\mathscr{O}([w_{\kappa_n}]_{cor}))}\rho_{\kappa_n}=\overline{\rho}. $$
However, choosing a converging subsequence (denoted the same way) $(\rho_{\kappa_n},\bu_{\kappa_n},w_{\kappa_n})$, and noting that
$$ [w_{\kappa_n}]_{cor}\to s \quad \text{ in } C_0^{0,1}(\Gamma),$$
as $n\to \infty$, we conclude as in previous section that
$$ \bu_{\kappa_n}\chi_{|\mathscr{O}([w_{\kappa_n}]_{cor})} \to \bu\chi_{|\mathscr{O}(s)}, \quad \text{ in } L^2(\mathbb{R}^3) $$
and
$$\rho_{\kappa_n}\chi_{|\mathscr{O}([w_{\kappa_n}]_{cor})} \rightharpoonup \rho\chi_{|\mathscr{O}(s)} = \overline{\rho}\chi_{|\mathscr{O}(s)}, \quad \text{ weakly* in } L^\infty(\mathbb{R}^3) $$
as $n\to \infty$, so passing to the limit in the continuity equation gives a clear contradiction with the boundary data $\rho_B < \overline{\rho}$. Therefore, there exists a $\beta>1$ such that
$$\frac{\beta}{|\mathscr{O}([w_{\kappa}]_{cor})|} \int_{\mathscr{O}([w_{\kappa}]_{cor}))}\rho_{\kappa}<\overline{\rho}, \quad \text{ for every } \kappa>0,$$
which yields, as in previous section, the estimates
\begin{eqnarray*}
\|p(\rho_\kappa)\|_{L^2(\mathscr{O}([w_\kappa]_{cor}))}\leq C((\beta-1)^{-1},\| \bu_B\|_{H^1(\mathscr{O}_{max})}), 
\end{eqnarray*}
and consequently
\begin{eqnarray*}
\kappa\|w_\kappa\|_{C^{0,1}(\Gamma)}\leq C\kappa\|w_\kappa\|_{H^{\frac72}(\Gamma)}\leq C((\beta-1)^{-1},\| \bu_B\|_{H^1(\mathscr{O}_{max})}). 
\end{eqnarray*}
Choosing $\kappa\geq \kappa_0$, where 
$$\kappa_0=4C((\beta-1)^{-1},\| \bu_B\|_{H^1(\mathscr{O}_{max})}) $$
yields
$$\|w_\kappa\|_{C^{0,1}(\Gamma)}\leq \frac1\kappa C((\beta-1)^{-1},\| \bu_B\|_{H^1(\mathscr{O}_{max})}) \leq \frac14, $$
and consequently $[w_\kappa]_{cor} = w_\kappa$, so for every $\kappa\geq \kappa_0$, we have that $(\rho_\kappa,\bb{u}_\kappa,w_\kappa)$ is a solution in the sense of Definition \ref{weak:sol}. This concludes the proof of existence of a weak solution.

\section{Acknowledgments}  B.M. was supported by: the European Union Next Generation EU through the National Recovery and Resilience Plan 2021--2026; an institutional grant from the University of Zagreb, Faculty of Science: project IK IA 1.1.3. Impact4Math; and by the Croatian Science Foundation under the project number IP-2019-04-5982.   The research of  \v S. N. has been supported by  Praemium Academiae of \v S. Ne\v casov\' a. The Institute of Mathematics, CAS is supported by RVO:67985840. The research of M.P. was partially supported by the Czech Science Foundation (GA\v{C}R), project No. 25-16592S. S.T. was supported by the Science Fund of the Republic of Serbia, GRANT No TF C1389-YF, Project title - FluidVarVisc. J.T.W. was partially supported by the National Science Foundation (USA). DMS-2307538.

\end{document}